\documentclass[reqno, a4paper, 11pt]{amsart}

\makeatletter
\let\origsection=\section
\def\section{\@ifstar{\origsection*}{\mysection}}
\def\mysection{\@startsection{section}{1}\z@{.7\linespacing\@plus\linespacing}{
.5\linespacing}{\normalfont\scshape\centering\S}}
\makeatother

\usepackage{comment}
\usepackage{amsmath,amssymb,amsthm}
\usepackage{mathrsfs}
\usepackage{dsfont}
\usepackage{mathabx}\changenotsign

\usepackage{microtype}
\usepackage{color}
\usepackage[utf8]{inputenc}
\usepackage{enumitem}
\newcommand\rmlabel{\upshape({\itshape\roman*\,\/})}
\newcommand\RMlabel{\upshape(\Roman*)}

\usepackage{xcolor}
\usepackage[backref=page]{hyperref}
\hypersetup{%
    colorlinks,
    linkcolor={red!60!black},
    citecolor={green!60!black},
    urlcolor={blue!60!black}
}

\usepackage{tikz}
\usepackage{subcaption}
\usetikzlibrary{calc,fit,shapes}
\usepackage{graphicx}
\graphicspath{{./Pictures/} }
\usepackage{bookmark}

\usepackage[abbrev,msc-links,backrefs]{amsrefs}
\usepackage{doi}

\renewcommand{\PrintDOI}[1]{\doi{#1}}

\usepackage[T1]{fontenc}
\usepackage{lmodern}
\usepackage[english]{babel}

\usepackage{fullpage}
\usepackage{setspace}
\usepackage{comment}
\usepackage{enumitem}

\newcommand{\sr}{\hat{r}}
\let\eps=\varepsilon
\let\polishlcross=\l
\def\l{\ifmmode\ell\else\polishlcross\fi}

\newcommand\qand{\quad\text{and}\quad}

\let\emptyset=\varnothing
\let\setminus=\smallsetminus

\newcommand{\EE}{\mathbb{E}}

\newcommand{\PP}{\mathbb{P}}

\newcommand{\oldqed}{}
\def\endofFact{\hfill\scalebox{.6}{$\Box$}}
\newenvironment{claimproof}[1][Proof]{
  \renewcommand{\oldqed}{\qedsymbol}
  \renewcommand{\qedsymbol}{\endofFact}
  \begin{proof}[#1]
}{
  \end{proof}
  \renewcommand{\qedsymbol}{\oldqed}
} 

\makeatletter
\def\moverlay{\mathpalette\mov@rlay}
\def\mov@rlay#1#2{\leavevmode\vtop{%
   \baselineskip\z@skip \lineskiplimit-\maxdimen
   \ialign{\hfil$\m@th#1##$\hfil\cr#2\crcr}}}
\newcommand{\charfusion}[3][\mathord]{
    #1{\ifx#1\mathop\vphantom{#2}\fi
        \mathpalette\mov@rlay{#2\cr#3}
      }
    \ifx#1\mathop\expandafter\displaylimits\fi}
\makeatother

\newtheorem{theorem}{Theorem}
\newtheorem{lemma}[theorem]{Lemma}
\newtheorem{corollary}[theorem]{Corollary}

\newtheorem{property}[theorem]{Property}
\newtheorem{claim}[theorem]{Claim}

\newtheorem{proposition}[theorem]{Proposition}
\newtheorem{definition}[theorem]{Definition}

\def\cP{\mathcal{P}}

\usepackage{datetime}
\usepackage{lineno}
\newcommand*\patchAmsMathEnvironmentForLineno[1]{%
\expandafter\let\csname old#1\expandafter\endcsname\csname #1\endcsname
\expandafter\let\csname oldend#1\expandafter\endcsname\csname end#1\endcsname
\renewenvironment{#1}%
{\linenomath\csname old#1\endcsname}%
{\csname oldend#1\endcsname\endlinenomath}}%
\newcommand*\patchBothAmsMathEnvironmentsForLineno[1]{%
\patchAmsMathEnvironmentForLineno{#1}%
\patchAmsMathEnvironmentForLineno{#1*}}%
\AtBeginDocument{%
\patchBothAmsMathEnvironmentsForLineno{equation}%
\patchBothAmsMathEnvironmentsForLineno{align}%
\patchBothAmsMathEnvironmentsForLineno{flalign}%
\patchBothAmsMathEnvironmentsForLineno{alignat}%
\patchBothAmsMathEnvironmentsForLineno{gather}%
\patchBothAmsMathEnvironmentsForLineno{multline}%
}

\makeatletter
\newcommand{\definetitlefootnote}[1]{%
  \newcommand\addtitlefootnote{%
    \makebox[0pt][l]{$^{*}$}%
    \footnote{\protect\@titlefootnotetext}
  }%
  \newcommand\@titlefootnotetext{\spaceskip=\z@skip $^{*}$#1}%
}
\makeatother

\begin{document}
\setstretch{1.18}
\shortdate
\yyyymmdddate
\settimeformat{ampmtime}

\definetitlefootnote{%
    An extended abstract of this work~\cite{berger19:_ramsey_EUROCOMB}
    will appear in the proceedings of EUROCOMB~2019.}

\title[The size-Ramsey number of powers of bounded degree trees]%
{The size-Ramsey number of powers of bounded degree
  trees\addtitlefootnote} 

\author[S.~Berger]{S\"{o}ren Berger}
\author[Y.~Kohayakawa]{Yoshiharu Kohayakawa}
\author[G.~S.~Maesaka]{Giulia Satiko Maesaka}
\author[T.~Martins]{Taísa Martins}
\author[W.~Mendon\c ca]{Walner Mendon\c ca}
\author[G.~O.~Mota]{Guilherme Oliveira Mota}
\author[O.~Parczyk]{Olaf Parczyk}

\address{Fachbereich Mathematik, Universit\"at Hamburg, Hamburg,
  Germany (S.~Berger and G.~S.~Maesaka)}
\email{\{\,soeren.berger\,|\,giulia.satiko.maesaka\,\}@uni-hamburg.de}

\address{Centro de Matem\'atica, Computa\c c\~ao e Cogni\c c\~ao,
  Universidade Federal do ABC, Santo Andr\'e, Brazil (G.~O.~Mota)} 
\email{g.mota@ufabc.edu.br}	

\address{Instituto de Matem\'atica e Estat\'{\i}stica, Universidade de
  S\~ao Paulo, Rua do Mat\~ao 1010, 05508-090 S\~ao Paulo, Brazil
  (Y.~Kohayakawa)} 
\email{yoshi@ime.usp.br}

\address{IMPA, Estrada Dona Castorina 110, Jardim Bot\^anico, Rio de
  Janeiro, RJ, Brazil (T.~Martins and W.~Mendon\c ca)}
\email{\{\,taisa.martins\,|\,walner\,\}@impa.br}

\address{Institut f\"ur Mathematik, Technische Universit\"at Ilmenau,
  Ilmenau, Germany (O.~Parczyk)}
\email{olaf.parczyk@tu-ilmenau.de}

\thanks{\rule[-.2\baselineskip]{0pt}{\baselineskip}%
  S.~Berger and G.~S.~Maesaka were partially supported by the European Research
  Concil (Consolidator grant PEPCo 724903).
  Y.~Kohayakawa was partially supported by CNPq (311412/2018-1, 423833/2018-9).
  T.~Martins and W.~Mendon\c{c}a were partially supported by CAPES.
  G.~O.~Mota was partially supported by FAPESP (2018/04876-1) and CNPq
  (304733/2017-2, 428385/2018-4). 
  O.~Parczyk was partially supported by the Carl Zeiss Foundation.
  The cooperation of the authors was supported by a joint CAPES/DAAD
  PROBRAL project (Proj.~430/15, 57350402, 57391197).
  This study was financed in part by CAPES, Coordenação de
  Aperfeiçoamento de Pessoal de Nível Superior, Brazil, Finance
  Code~001.
  FAPESP is the S\~ao Paulo Research Foundation.  CNPq is the National
  Council for Scientific and Technological Development of
  Brazil.%
}

\begin{abstract}
  Given a positive integer~$s$, the \textit{$s$-colour size-Ramsey
    number} of a graph~$H$ is the smallest integer~$m$ such that there
  exists a graph~$G$ with~$m$ edges with the property that, in any
  colouring of~$E(G)$ with~$s$ colours, there is a monochromatic copy
  of~$H$.  We prove that, for any positive integers~$k$ and~$s$, the
  $s$-colour size-Ramsey number of the $k$th power of any $n$-vertex
  bounded degree tree is linear in~$n$.  As a corollary we obtain that
  the $s$-colour size-Ramsey number of $n$-vertex graphs with bounded
  treewidth and bounded degree is linear in $n$, which answers a
  question raised by Kam\v{c}ev, Liebenau, Wood and
  Yepremyan~[\emph{The size Ramsey number of graphs with bounded
    treewidth}, arXiv:1906.09185 (2019)].
\end{abstract}


\maketitle

\section{Introduction}
\label{sec:introduction}
Given graphs $G$ and $H$ and a positive integer~$s$, we denote by
$G\to (H)_s$ the property that any $s$-colouring of the edges of $G$
contains a monochromatic copy of $H$.  We are interested in the
problem proposed by Erd\H{o}s, Faudree, Rousseau and
Schelp~\cite{ErFaRoSc78} of determining the minimum integer~$m$ for
which there is a graph $G$ with $m$ edges such that property
$G\to (H)_2$ holds.  Formally, the \emph{$s$-colour size-Ramsey
  number} $\sr_s(H)$ of a graph $H$ is defined as follows:
\begin{equation*}
  \sr_s(H)=\min\{e(G) \colon G\rightarrow (H)_s\}.
\end{equation*}

Answering a question posed by Erd\H{o}s~\cite{erdHos1981combinatorial},
Beck~\cite{Be83} showed that $\sr_2(P_n)=O(n)$ by means of a probabilistic
proof.
Alon and Chung~\cite{AlCh88} proved the same fact by explicitly constructing a
graph~$G$ with $O(n)$ edges such that $G\to (P_n)_2$.
In the last decades many successive improvements were obtained in
order to determine the size-Ramsey number of paths (see, e.g.,
\cites{Be83, bollobas1986extremal, dudek2017some} for lower
bounds, and \cites{Be83, DuPr15,
  letzter16:_path_ramsey, dudek2017some} for upper bounds).
The best known bounds for paths are
$5n/2 - 15/2\leq\sr_2(P_n)\leq 74n$ from~\cite{dudek2017some}.  For
any $s\geq 2$ colours, Dudek and Pra{\l}at~\cite{dudek2017some} and
Krivelevich~\cite{krivelevich2017long} proved that there are positive
constants~$c$ and~$C$ such that
$cs^2n\leq\sr_s(P_n)\leq Cs^2(\log s)n$.

Moving away from paths, Beck~\cite{Be83} asked whether $\sr_2(H)$ is
linear for any bounded degree graph. This question was later answered
negatively by R\"odl and Szemer\'edi~\cite{RoSz00}, who constructed a
family $\{H_n\}_{n\in \mathbb{N}}$ of $n$-vertex graphs of maximum
degree~$\Delta(H_n)\leq3$ such that $\sr_2(H_n)=\Omega(n\log ^{1/60}n)$.
The current best upper bound for the size-Ramsey number of graphs with
bounded degree was obtained in~\cite{KoRoScSz11} by Kohayakawa,
R\"{o}dl, Schacht and Szemer\'{e}di, who proved that for any positive
integer $\Delta$ there is a constant~$c$ such that, for any graph $H$
with $n$ vertices and maximum degree~$\Delta$, we have
\begin{equation*}
  \sr_2(H)\leq cn^{2-1/\Delta}\log^{1/\Delta}n.
\end{equation*}
For more results on the size-Ramsey number of bounded degree graphs
see~\cites{De12,FrPi87,HaKo95,HaKoLu95,Ke93, kohayakawa07+:_ramsey}.

Let us turn our attention to powers of bounded degree graphs.  Let~$H$
be a graph with~$n$ vertices and let~$k$ be a positive integer.  The
\emph{$k$th power}~$H^k$ of~$H$ is the graph with vertex set~$V(H)$ in
which there is an edge between distinct vertices~$u$ and~$v$ if and
only if~$u$ and~$v$ are at distance at most~$k$ in~$H$.  Recently it
was proved that the $2$-colour size-Ramsey number of powers of paths
and cycles is linear~\cite{2colourSizeRamsey}.  This result was
extended to any fixed number~$s$ of colours in~\cite{CYKMMRB}, i.e.,
\begin{equation}\label{eq:mainpathcycle}
  \sr_s(P_n^k)=O_{k,s}(n)\quad\text{and}\quad\sr_s(C_n^k)=O_{k,s}(n).
\end{equation}
In our main result (Theorem~\ref{thm:main}) we extend~\eqref{eq:mainpathcycle}
 to bounded powers of bounded degree trees.
 We prove that for any positive integers~$k$ and~$s$, the $s$-colour
size-Ramsey number of the $k$th power of any $n$-vertex bounded degree
tree is linear in~$n$.

\begin{theorem}
  \label{thm:main}
  For any positive integers $k$, $\Delta$ and~$s$ and any $n$-vertex
  tree~$T$ with $\Delta(T)\leq \Delta$, we have
  \begin{align*}
    \hat{r}_s(T^k) =O_{k,\Delta,s}(n).
  \end{align*}
\end{theorem}

We remark that Theorem~\ref{thm:main} is equivalent to the following result
for the `general' or `off-diagonal' size-Ramsey number
$\hat{r}(H_1,\dots,H_s)=\min\{e(G)\colon G\to(H_1,\dots,H_s)\}$: if
$H_i=T_i^k$ for $i=1,\dots,s$ where $T_1,\dots,T_s$ are bounded degree
trees, then $\hat{r}(H_1,\dots,H_s)$ is linear
in~$\max_{1\leq i\leq s} v(H_i)$.  To see this, it is sufficient to
apply Theorem~\ref{thm:main} to a tree containing the disjoint 
union of $T_1,\dots,T_s$.

The graph that we present to prove Theorem~\ref{thm:main} does not
depend on $T$, but only on~$\Delta$, $k$ and~$n$.  Moreover, our proof
not only gives a monochromatic copy of~$T^k$ for a given~$T$, but a
monochromatic subgraph that contains a copy of the $k$th power of
every $n$-vertex tree with maximum degree at most~$\Delta$.  That is,
we prove the existence of so called `partition universal graphs'
with~$O_{k,\Delta,s}(n)$ edges for the family of powers~$T^k$ of
$n$-vertex trees with~$\Delta(T)\leq\Delta$.

Theorem~\ref{thm:main} was announced in the extended
abstract~\cite{berger19:_ramsey_EUROCOMB}.  While finalizing this
paper, we learned that Kam\v{c}ev, Liebenau, Wood, and
Yepremyan~\cite{KLWY_sizeRamsey} proved, among other things, that the
$2$-colour size-Ramsey number of an $n$-vertex graph with bounded
degree and bounded treewidth is $O(n)$\footnote{They in fact formulate
  this for the general $2$-colour size-Ramsey
  number~$\hat{r}(H_1,H_2)$.}.  This is equivalent to our result for
$s=2$.  Indeed, any graph with bounded treewidth and bounded maximum
degree is contained in a suitable blow-up of some bounded degree
tree~\cites{ding95:_some,Wood_tpw} and a blow-up of a bounded degree
tree is contained in the power of another bounded degree tree.
Conversely, bounded powers of bounded degree trees have bounded
treewidth and bounded degree.  Therefore, we obtain the following
equivalent version of Theorem~\ref{thm:main}, which generalises the
result from~\cite{KLWY_sizeRamsey} and answers one of their main open
questions (Question~17 in~\cite{KLWY_sizeRamsey}).

\begin{corollary}
  \label{cor:main}
  For any positive integers $k$, $\Delta$ and~$s$ and any $n$-vertex
  graph~$H$ with treewidth~$k$ and $\Delta(H)\leq \Delta$, we have
  \begin{align*}
    \hat{r}_s(H) =O_{k,\Delta,s}(n).
  \end{align*}
\end{corollary}

The proof of Theorem~\ref{thm:main} follows the strategy developed
in~\cite{CYKMMRB}, proving the result by induction on the number of
colours~$s$.  Very roughly speaking, we start with a graph~$G$ with
suitable properties and, given any $s$-colouring of the edges of~$G$
($s\geq2$), either we obtain a monochromatic copy of the power of the
desired tree in~$G$, or we obtain a large subgraph~$H$ of~$G$ that is
coloured with at most $s-1$ colours; moreover, the graph~$H$ that we
obtain is such that we can apply the induction hypothesis on it.
Naturally, we design the requirements on our graphs in such a way that
this induction goes through.  As it turns out, the graph~$G$ will be a
certain blow-up of a random-like graph.
While this approach seems uncomplicated upon first glance, the proof
requires a variety of additional ideas and technical details.

To implement the above strategy, we need, among other results, two new
and key ingredients which are interesting on their own: (\textit{i})~a
result that states that for any sufficiently large graph~$G$, either
$G$ contains a large expanding subgraph or there is a given number of
reasonably large disjoint subsets of~$V(G)$ without any edge between
any two of them (see Lemma~\ref{lem:alternatives}\footnote{We are
  grateful to the authors of~\cite{KLWY_sizeRamsey}, who pointed out
  to us that similar lemmas have been proved
  in~\cites{pokrovskiy17:_ramsey, pokrovskiy18:_ramsey}.});
(\textit{ii})~an embedding result that states that in order to embed a
power~$T^k$ of a tree~$T$ in a certain blow-up of a graph~$G$ it is
enough to find an embedding of an auxiliary tree~$T'$ in~$G$ (see
Lemma~\ref{lem:bluepower}).

\section{Auxiliary results}
\label{sec:prelim}

In this section we state a few results which will be needed in the proof 
of our main theorem.
The first lemma guarantees that, in a graph $G$ that have edges between large
subsets of vertices, there exists a long ``transversal'' path along
a constant number of large subsets of vertices of $G$.

\begin{lemma}[\cite{2colourSizeRamsey}*{Lemma~3.5}]\label{lem:H2}
For every integer $\ell \geq 1$ and every $\gamma>0$ there exists
$d_0=2+4/(\gamma(\ell +1))$ such that the following holds for any $d\ge d_0$.
Let $G$ be a graph on $dn$ vertices such that for every pair of disjoint sets
$X,Y\subseteq V(G)$ with $|X|, |Y|\geq \gamma n$ we have $e_{G}(X,Y)> 0$.
Then for every family $V_1,\dots,V_{\ell}\subseteq V(G)$ of pairwise disjoint sets
each of size at least $\gamma d n$, there is a path $P_n=(x_1,\dots,x_{n})$
in~$G$ with $x_i \in V_j$ for all~$1\leq i\leq n$, where $j\equiv i\pmod{\ell}$.
\end{lemma}

We will also use the classical Chernoff's inequality and
K\H{o}v\'ari--S\'{o}s--Tur\'{a}n theorem.

\begin{theorem}[Chernoff's inequality]\label{thm:chernoff}
Let $0<\eps\leq 3/2$.
If $X$ is a sum of independent Bernoulli random variables then
\begin{equation*}
\PP(|X-\EE[X]| > \eps \EE[X]) \leq 2\cdot e^{-(\eps^2/3)\EE[X]}\,.
\end{equation*}
\end{theorem}

\begin{theorem}[K\H{o}v\'ari--S\'{o}s--Tur\'{a}n~\cite{KoSoTu54}]
	\label{thm:kst}
	Let $k\geq1$ and let~$G$ be a bipartite graph with~$x$ vertices in each
	vertex class.  If~$G$ contains no~$K_{2k,2k}$, then~$G$ has at
	most~$4x^{2-1/(2k)}$ edges.
\end{theorem}

\section{Bijumbledness, expansion and embedding of trees}
\label{sec:biexem}

In this section we provide the necessary tools to obtain the desired
monochromatic embedding of a power of a tree in the
proof of Theorem~\ref{thm:main}.
We start by defining the expanding property of a graph.

\begin{property}	[Expanding]
A graph~$G$ is \emph{$(n,a,b)$-expanding} if for all
$X\subseteq V(G)$ with $|X|\leq a(n-1)$, we have $|N_G(X)|\geq b|X|$.	
\end{property}

Here $N_G(X)$ is the set of neighbours of $X$, i.e., all vertices in $V(G)$ that
share an edge with some vertex from $X$.  The following embedding result due to
Friedman and Pippenger~\cite{FrPi87} guarantees the existence of copies of
bounded degree trees in expanding graphs.

\begin{lemma}
  \label{lem:FP}
  Let $n$ and $\Delta$ be positive integers and $G$ a non-empty graph.
  If $G$ is $(n,2,\Delta+1)$-expanding, then~$G$ contains any
  $n$-vertex tree with maximum degree~$\Delta$ as a subgraph.
\end{lemma}

Owing to Lemma~\ref{lem:FP}, we are interested in graph properties
that guarantee expansion.  One such property is bijumbledness, defined
below.  Denote by $e_G(X,Y)$ the number of edges between two disjoint
sets~$X$ and~$Y$ in a graph~$G$.

\begin{property}[Bijumbledness]
A graph~$G$ on~$N$ vertices is \emph{$(p,\theta)$-bijumbled} if, for all
disjoint sets $X$ and~$Y\subseteq V(G)$ with $|X|\leq|Y|\leq pN|X|$, we
have $\big|e_G(X,Y)-p|X||Y|\big|\leq\theta\sqrt{|X||Y|}$. 
\end{property}

Note that bijumbledness immediately implies that 
\begin{align}
\label{def:P-sets}
\text{for all disjoint $X,Y\subseteq V(G)$ with $|X|, |Y| > \theta/p$ we have
	$e_{G}(X,Y)> 0$.}
\end{align}
The following basic proposition, which can be proved by an averaging argument,
guarantees that in bijumbled graphs the number of edges inside subsets of
vertices is also controlled.

\begin{proposition}\label{prop:jumbled}
Let $G$ be a $(p,\theta)$-bijumbled graph with $N$ vertices.
Then, for any $U\subseteq V(G)$ we have
\begin{equation*}
  \left|e(U) - p\binom{|U|}{2}\right| \leq \theta |U|.
\end{equation*}
\end{proposition}

We now state the first main novel ingredient in the proof of our main
result (Theorem~\ref{thm:main}). The following lemma ensures that in a
sufficiently large graph we either get an expanding subgraph with
appropriate parameters or we get reasonably large disjoint subsets of
vertices that span no edges between them.  This result was inspired
by~\cite{Po16}*{Theorem~1.5}.  Furthermore, we remark that similar
results have been proved in~\cites{pokrovskiy17:_ramsey,
  pokrovskiy18:_ramsey}. 

\begin{lemma}
  \label{lem:alternatives}
  Let~$f \geq 0$, $D \geq 0$, $\ell \geq 2$ and~$\eta > 0$ be given and let~$A
  = (\ell - 1)(D+1)(\eta + f) + \eta$.

  If~$G$ is a graph on at least~$An$ vertices, then
  \begin{enumerate}[label=\rmlabel]
    \item \label{prop:blue}
      either there is a non-empty set $Z \subseteq V(G)$
      such that $G[Z]$ is $(n,f,D)$-expanding,
    \item \label{prop:gray} or there exist $V_1,\dots,V_\ell \subseteq V(G)$
      such that $|V_i| \ge  \eta n$ for $1\leq i\leq  \ell$ and $e(V_i,V_j) =
      0$ for~$1\leq i<j\leq \ell$. 
  \end{enumerate}
\end{lemma}

\begin{proof}
  Let us assume that \ref{prop:blue} does not hold. 
  Since~$G$ is not $(n,f,D)$-expanding, we can take~$V_1 \subseteq V(G)$ of
  maximum size satisfying that~$|V_1| \leq (\eta + f)n$ and~$|N_G(V_1)| <
  D|V_1|$. 
  We claim that~$|V_1| \geq \eta n$. Assume, for the sake of contradiction
  that~$|V_1| < \eta n$. Let
  \[
    W_1 = V(G) \setminus (V_1 \cup N_G(V_1)).
  \]
  Then $|W_1| > An - (D + 1)\eta n > 0$. Applying that~\ref{prop:blue} does not
  hold, we get~$X \subseteq W_1$ such that~$|X| \leq f(n - 1)$
  and~$|N_{G[W_1]}(X)| < D|X|$. Note that~$N_{G}(X) \subseteq N_{G[W_1]}(X)
  \cup N_G(V_1)$. Thus
  \begin{align*}
    |N_G(X \dot{\cup} V_1)| 
    & =  |N_{G[W_1]}(X) \cup N_G(V_1)| \\
    & < D(|X| + |V_1|).
  \end{align*}
  Also~$|X \dot{\cup} V_1| \leq (\eta + f)n$, deriving a contradiction to the
  maximality of~$V_1$.

  Let~$1 \leq k \leq \ell -2$ and suppose we have~$(V_1, \dots, V_{k})$ such
  that 
  \begin{enumerate}[label=\RMlabel]
    \item \label{prop:size} $|V_i| \geq \eta n$, for $1 \leq i \leq k$;
    \item \label {prop:disjoint} $e(V_i, V_j) = 0$, for~$1 \leq i < j \leq k$;
    \item \label {prop:neighbour} $|\bigcup_{i = 1}^{k}(V_i \cup N_G(V_i))| <
      k(D+1)(\eta + f) n$.
  \end{enumerate}
  We can increase this sequence in the following way. Let $W_k = V(G)
  \setminus \bigcup_{i = 1}^{k} (V_i \cup N_G(V_i))$ and note that
  \begin{align*}
    |W_k| 
    & \overset{\text{\ref{prop:neighbour}}}{\geq} An - (\ell - 2)(D+1)(\eta
    + f)n \\ 
    & \geq  (D+1)(\eta + f)n + \eta n \\ 
    & > 0.
  \end{align*}
  Since~\ref{prop:blue} does not hold, there exists $V_{k+1} \subseteq W_k$ of
  maximum size with $|V_{k+1}| \leq (\eta + f)n$ such
  that~$|N_{G[W_k]}(V_{k+1})| < D|V_{k+1}|$. Note that~$e(V_i, V_{k+1}) \leq
  e(V_i, W_{k+1}) = 0$, for every~$1 \leq i \leq k$. Therefore we have
  that~\ref{prop:disjoint} holds for the sequence $(V_1,\ldots,V_{k+1})$.
  Furthermore, note that
  \begin{align}
    N_G(V_{k+1}) \subseteq \bigcup_{i = 1}^{k} N_G(V_i) \cup
    N_{G[W_k]}(V_{k+1})\,. \label{eq:neighG}
  \end{align}
  This gives us~\ref{prop:neighbour} for the sequence $(V_1,\ldots,V_{k+1})$,
  since
  \begin{align*}
    \left| \bigcup_{i = 1}^{k+1}(V_i \cup N_G(V_i)) \right| 
    & \overset{\eqref{eq:neighG}}{=} \left|\bigcup_{i = 1}^{k}(V_i \cup
    N_G(V_i)) \cup V_{k+1} \cup N_{G[W_{k}]}(V_{k+1}) \right| \\
    & < (k+1)(D+1)(\eta + f) n.
  \end{align*}
  To see that $(V_1,\ldots,V_{k+1})$ satisfies \ref{prop:size}, define 
  \begin{equation*}
    W_{k+1} 
     = V(G) \setminus \bigcup_{i = 1}^{k+1} (V_i \cup N_G(V_i)) 
    \overset{\eqref{eq:neighG}}{=}  W_k \setminus (V_{k+1} \cup
    N_{G[W_k]}(V_{k+1})).
  \end{equation*}
  Assume that $|V_{k+1}| < \eta n$ and derive a contradiction as before.

  Therefore, when~$k = \ell - 2$, we generate a sequence~$(V_1, \dots, V_{\ell
  - 1})$ with the properties required by~\ref{prop:gray}. To complete the
  sequence, note that~\ref{prop:neighbour} gives that~$|W_{\ell - 1}| \geq \eta
  n$ and set~$V_{\ell} = W_{\ell - 1}$.

\end{proof}

As a corollary of the previous lemma, we get the following lemma that says that
sufficiently large bijumbled graphs contain a non-empty expanding subgraph.

\begin{lemma}[Bijumbledness implies expansion] \label{lem:jumb-uni-exp} Let
  $f$, $\theta$, $D$ and $c\geq 1$ be positive numbers with $c \geq
  4(D+2) \theta$ and $a\geq 2(D+1) f$. If $G$ is a $(c/(an),\theta)$-bijumbled
  graph with~$an$ vertices, then there exists a non-empty subgraph~$H$ of~$G$
  that is $(n,f,D)$-expanding.  
\end{lemma}
\begin{proof}
  Let $G$ be a $(c/(an),\theta)$-bijumbled graph. Suppose that no subgraph of $G$ is $(n,f,D)$-expanding.
  We apply Lemma~\ref{lem:alternatives} with $\ell=2$ and $\eta = \frac{2 \theta a}{c}$. Since $a \ge (D+1)(f+\eta)+\eta$, we get two disjoint sets $V_1,V_2 \subseteq V(G)$ with $|V_1|=|V_2|=\eta n$ such that $e_G(V_1,V_2)=0$.
  On the other hand, by~\eqref{def:P-sets}, we have $e_G(V_1,V_2)>0$. Therefore, there is some subgraph of $G$ that is $(n,f,D)$-expanding.
\end{proof}

The next lemma is crucial for embedding the desired power of a tree.
Let $G$ be a graph and~$\ell \ge r$ be positive integers.  An
\emph{$(\ell,r)$-blow-up} of $G$ is a graph obtained from $G$ by
replacing each vertex of $G$ by a clique of size $\ell$ and for every
edge of $G$ arbitrarily adding a complete bipartite graph $K_{r,r}$
between the cliques corresponding to the vertices of this edge.

\begin{lemma}[Embedding lemma for powers of trees]
	\label{lem:bluepower}
	Given positive integers $k$ and ~$\Delta$, there exists $r_0$ such that the following holds for every $n$-vertex tree~$T$ with maximum degree $\Delta$.
	There is a tree~$T^\prime= T^\prime(T, k)$ on at most~$n+1$ vertices and with maximum
	degree at most $\Delta^{2k}$ such that for every graph $J$ with $T^\prime \subseteq J$ and any $(\ell,r)$-blow-up $J^\prime$ of $J$ with~$\ell \geq r \ge r_0$ we have~$T^k\subseteq J^\prime$.
	
\end{lemma}

\begin{proof}
  Given positive integers $k$,~$\Delta$, take~$r_0 = \Delta^{4k}$.  Let $T$ be a
  $n$-vertex tree with maximum degree $\Delta$.  Let~$x_0$ be any vertex in
  $V(T)$ and consider $T$ as rooted at $x_0$.  For each vertex $v \in V(T)$,
  let $D(v)$ denote the set of \emph{descendants} of $v$ in $T$ (including $v$
  it self). Let $D^{i}(v)$ be the set of vertices $u \in D(v)$ at distance at
  most $i$ from $v$ in $T$.

  Let $T'$ be a tree with vertex set consisting of a special vertex $x^*$ and
  the vertices $x \in V(T)$ such that the distance between $x$ and $x_0$ is a
  multiple of $2k$. The edge set of $T'$ consists of the edge $x^*x_0$ and the
  pair of vertices $x, y \in V(T')\setminus\{x^*\}$ for which $x \in D^{2k}(y)$
  or $y \in D^{2k}(x)$.  That is, 
  \begin{align*}
    V(T^\prime) & = \{  x \in V(T) \colon \mathrm{dist}_T(x_0,x) \equiv
    0\;(\bmod\; 2k) \} \cup \{ x^* \} \\
    E(T^\prime) & = \left\{ xy \in \binom{V(T^\prime)\setminus\{x^*\}}{2}
    \colon x \in D^{2k}(y) \text{ or } y \in D^{2k}(x) \right\} \cup \{ x^*x_0 \}.
  \end{align*}
  In particular, note that $\Delta(T') \leq \Delta^{2k}$ and $|V(T')| \leq
  n+1$.  Let us consider $T^\prime$ as a tree rooted at~$x^*$. 	

  Now suppose that~$J$ is a graph such that~$T'\subseteq J$ and $J'$ is an
  $(\ell,r)$-blow up of $J$ with $\ell\geq r \geq r_0$. Our goal is to show
  that $T^k \subseteq J'$. First, since $J'$ is an $(\ell,r)$-blow up of $J$,
  there is a collection $\{K(x) : x \in V(J)\}$ of disjoint $\ell$-cliques in
  $J'$ such that for each edge $xy\in E(J)$, there is a copy of $K_{r,r}$
  between the vertices of $K(x)$ and $K(y)$. Let us denote by $K(x,y)$ such
  copy of $K_{r,r}$.

  \begin{figure}
  	\begin{subfigure}{.5\linewidth}
  		\centering
  		\scalebox{.8}{
  			
  			\begin{tikzpicture}
  			[auto,
  			vertex/.style={minimum size=2pt,fill,draw,circle},
  			sibling distance=3cm,level distance=1.7cm]
  			
  			\node (x0) [vertex,label=right:$x_0$] {}
  			child { node (1) [vertex,label=right:$x_{1}$] {}
  				child { node (3) [vertex,label=right:$x_{3}$] {}
  					child { node (6) [vertex,label=right:$x_{6}$] {}
  						child { node (9) [vertex,label=right:$x_{9}$] {}
  							child { node (12) [vertex,label=right:$x_{12}$] {}
  								child { node (15) [vertex,label=right:$x_{15}$] {}
  									child { node (18) [vertex,label=right:$x_{18}$] {}
  										child { node (20) [vertex,label=right:$x_{20}$] {}
  										}
  									}
  								}
  							}
  						}
  					}
  					child { node (7) [vertex,label=right:$x_{7}$] {}
  						child { node (10) [vertex,label=right:$x_{10}$] {}
  						}
  					}
  				}
  				child { node (4) [vertex,label=right:$x_{4}$] {}
  				}
  			}
  			child { node (2) [vertex,label=right:$x_{2}$] {}
  				child { node (5) [vertex,label=right:$x_{5}$] {}
  					child { node (8) [vertex,label=right:$x_{8}$] {}
  						child { node (11) [vertex,label=right:$x_{11}$] {}
  							child { node (13) [vertex,label=right:$x_{13}$] {}
  								child { node (16) [vertex,label=right:$x_{16}$] {}
  									child { node (19) [vertex,label=right:$x_{19}$] {}
  									}
  								}
  							}
  							child { node (14) [vertex,label=right:$x_{14}$] {}
  								child { node (17) [vertex,label=right:$x_{17}$] {}
  								}
  							}
  						}
  					}
  				}
  			};
  			
  			\node (w2) at ($(2)+(14pt,0)$){};
  			\node [fit=(x0)(1)(w2),rounded corners,draw,label=left:$D^{+}(x_0)$] {};
  			\node (w8) at ($(8)+(14pt,0)$){};
  			\node [fit=(3)(4)(5)(6)(7)(w8),rounded corners,draw=red,label=left:$D^{-}(x_0)$] {};
  			\node (w12) at ($(12)+(17pt,0)$){};
  			\node [fit=(9)(w12),rounded corners,draw,label=left:$D^{+}(x_{9})$] {};
  			\node (w10) at ($(10)+(17pt,0)$){};
  			\node [fit=(10)(w10),rounded corners,draw,label=below:$D^{+}(x_{10})$] {};
  			\node (w14) at ($(14)+(17pt,0)$){};
  			\node [fit=(11)(13)(w14),rounded corners,draw,label=below:$D^{+}(x_{11})$] {};
  			\node (w18) at ($(18)+(17pt,0)$){};
  			\node [fit=(15)(w18),rounded corners,draw=red,label=left:$D^{-}(x_{9})$] {};
  			\node (w20) at ($(20)+(17pt,0)$){};
  			\node [fit=(20)(w20),rounded corners,draw,label=left:$D^{+}(x_{20})$] {};
  			\node (w17) at ($(17)+(17pt,0)$){};
  			\node [fit=(16)(17)(19)(w17),rounded corners,draw=red,label=below:$D^{-}(x_{11})$] {};
  			\end{tikzpicture}
  			
  		}
  		\caption{Tree $T$.}
  		\label{fig:sub1}
  	\end{subfigure}%
  	\begin{subfigure}{.5\linewidth}
  		\centering
  		\scalebox{.8}{

  			\begin{tikzpicture}
  			[auto,
  			vertex/.style={minimum size=2pt,fill,draw,circle},
  			sibling distance=1.5cm,level distance=1cm]
  			
  			\node (xs) [vertex,label=left:$x^{\star}$] {}
  			child { node (x0) [vertex,label=left:$x_0$] {}
  				child { node (9) [vertex,label=left:$x_9$] {}
  					child{ node (20) [vertex,label=left:$x_{20}$] {}
  					}
  				}
  				child { node (10) [vertex,label=below:$x_{10}$] {}
  				}
  				child { node (11) [vertex,label=below:$x_{11}$] {}
  				}
  			};
  			\end{tikzpicture}
  			
  		}
  		\caption{Corresponding $T'$.}
  		\label{fig:sub2}
  	\end{subfigure}\\[1ex]
  	\begin{subfigure}{\linewidth}
  		\centering
  		\scalebox{.8}{
  			
  			\begin{tikzpicture}
  			[auto,
  			vertex/.style={circle,fill,draw},
  			rec/.style={shape=ellipse, anchor=center, draw}]
  			
  			\def\r{.5}
  			\def\s{.5cm}
  			
  			\node (0) at (0,0) {};
  			\node (1) at ($(0*\r,-3.5*\r)$) {};
  			\node (2) at ($(-5.5*\r,-6.5*\r)$) {};
  			\node (3) at ($(5.5*\r,-6.5*\r)$) {};
  			\node (4) at ($(0*\r,-7*\r)$) {};
  			\node (5) at ($(-9*\r,-11*\r)$) {};
  			\node (6) at ($(0*\r,-11*\r)$) {};
  			\node (7) at ($(9*\r,-11*\r)$) {};
  			\node (8) at ($(-9*\r,-16*\r)$) {};
  			\node (9) at ($(-9*\r,-20*\r)$) {};
  			\node (b) at ($(0*\r,1*\r)$) {};
  			\node (c) at ($(0*\r,-5.25*\r)$) {};
  			\node (d) at ($(-9*\r,-13.5*\r)$) {};
  			\node (e) at ($(0*\r,-12*\r)$) {};
  			\node (f) at ($(9*\r,-12*\r)$) {};
  			\node (g) at ($(-9*\r,-21*\r)$) {};
  			
  			\node [rec, minimum width=2*\s, minimum height=2*\s, label=above:$D^{+}(x_0)$] at (0){};
  			\node [rec, minimum width=2*\s, minimum height=2*\s, label=right:$D^{-}(x_0)$] at (1){};
  			\node [rec, minimum width=2*\s, minimum height=2*\s, label=above:$D^{+}(x_9)$] at (2){};
  			\node [rec, minimum width=2*\s, minimum height=2*\s, label=above:$D^{+}(x_{11})$] at (3){};
  			\node [rec, minimum width=2*\s, minimum height=2*\s, label=above:$D^{+}(x_{10})$] at (4){};
  			\node [rec, minimum width=2*\s, minimum height=2*\s, label=below:$D^{-}(x_9)$] at (5){};
  			\node [rec, minimum width=2*\s, minimum height=2*\s, label=below:$D^{-}(x_{10})$] at (6){};
  			\node [rec, minimum width=2*\s, minimum height=2*\s, label=below:$D^{-}(x_{11})$] at (7){};
  			\node [rec, minimum width=2*\s, minimum height=2*\s, label=above:$D^{+}(x_{20})$] at (8){};
  			\node [rec, minimum width=2*\s, minimum height=2*\s, label=below:$D^{-}(x_{20})$] at (9){};
  			\node [rec, minimum width=6.3*\s, minimum height=4.5*\s, label=left:$K(x_{\star})$] at (b){};
  			\node [rec, minimum width=20*\s, minimum height=7*\s, label=left:$K(x_{0})$] at (c){};
  			\node [rec, minimum width=8.5*\s, minimum height=7.5*\s, label=left:$K(x_{9})$] at (d){};
  			\node [rec, minimum width=8.5*\s, minimum height=4.5*\s, label=below:$K(x_{10})$] at (e){};
  			\node [rec, minimum width=8.5*\s, minimum height=4.5*\s, label=below:$K(x_{11})$] at (f){};
  			\node [rec, minimum width=8.5*\s, minimum height=4.5*\s, label=left:$K(x_{20})$] at (g){};
  			
  			\pgfmathtruncatemacro\n{6}
  			\pgfmathtruncatemacro\ang{360/\n}
  			
  			\foreach \j in {0,...,9}{
  				\foreach \i in {1,...,\n}{
  					\node (\j\i) at ($(\j)+(10+\i*\ang:.75*\s)$) {};
  				}
  			}
  			
  			\foreach \i in {1,...,\n}{
  				\foreach \j in {1,...,\n}{
  					\draw (0\i) edge (1\j);
  				}
  			}
  			
  			\foreach \i in {1,...,\n}{
  				\foreach \j in {1,...,\n}{
  					\draw (2\i) edge (5\j);
  				}
  			}
  			
  			\foreach \i in {1,...,\n}{
  				\foreach \j in {1,...,\n}{
  					\draw (4\i) edge (6\j);
  				}
  			}
  			
  			\foreach \i in {1,...,\n}{
  				\foreach \j in {1,...,\n}{
  					\draw (3\i) edge (7\j);
  				}
  			}
  			
  			\foreach \i in {1,...,\n}{
  				\foreach \j in {1,...,\n}{
  					\draw (8\i) edge (9\j);
  				}
  			}
  			\end{tikzpicture}
  			
  		}
  		\caption{Embedding $T^k$ into an $(\ell,r)$-blow-up of $T'$.}
  		\label{fig:sub3}
  	\end{subfigure}
  	\caption{Illustration of the concepts and notation used throughout the
  		proof of Lemma~\ref{lem:bluepower} when $\Delta = 3$ and $k =2$.}
  	\label{fig1}
  \end{figure}
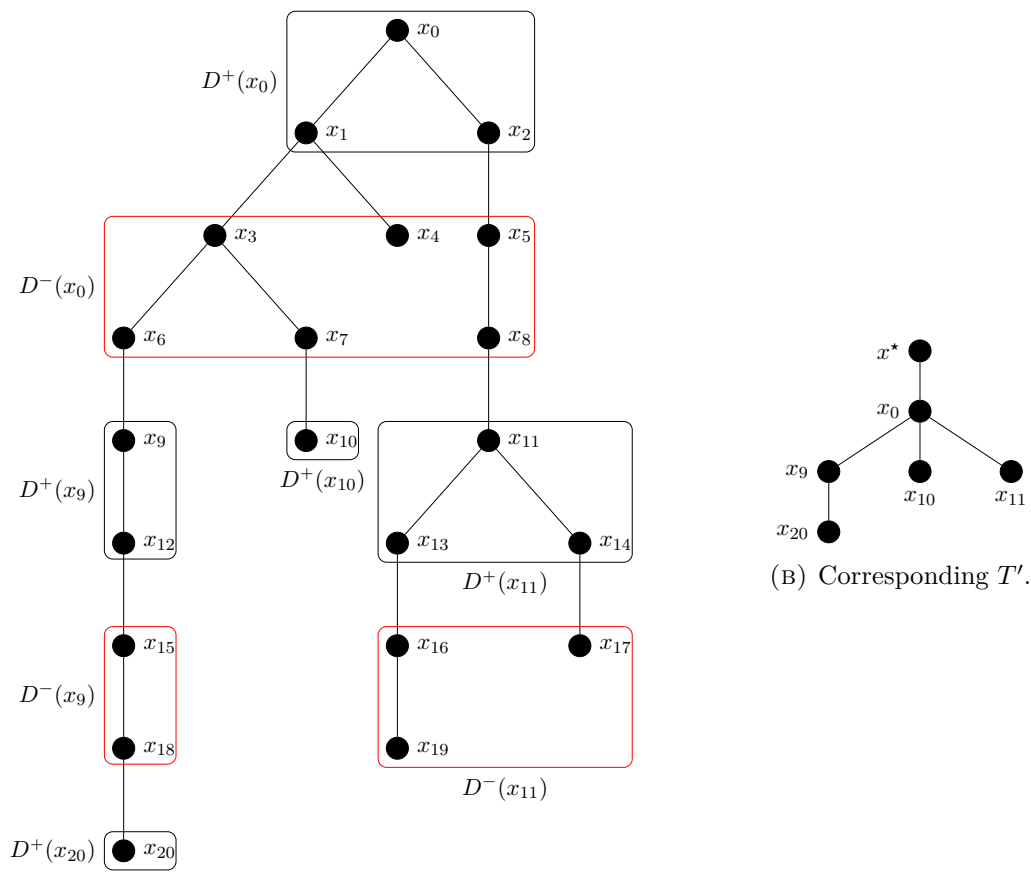
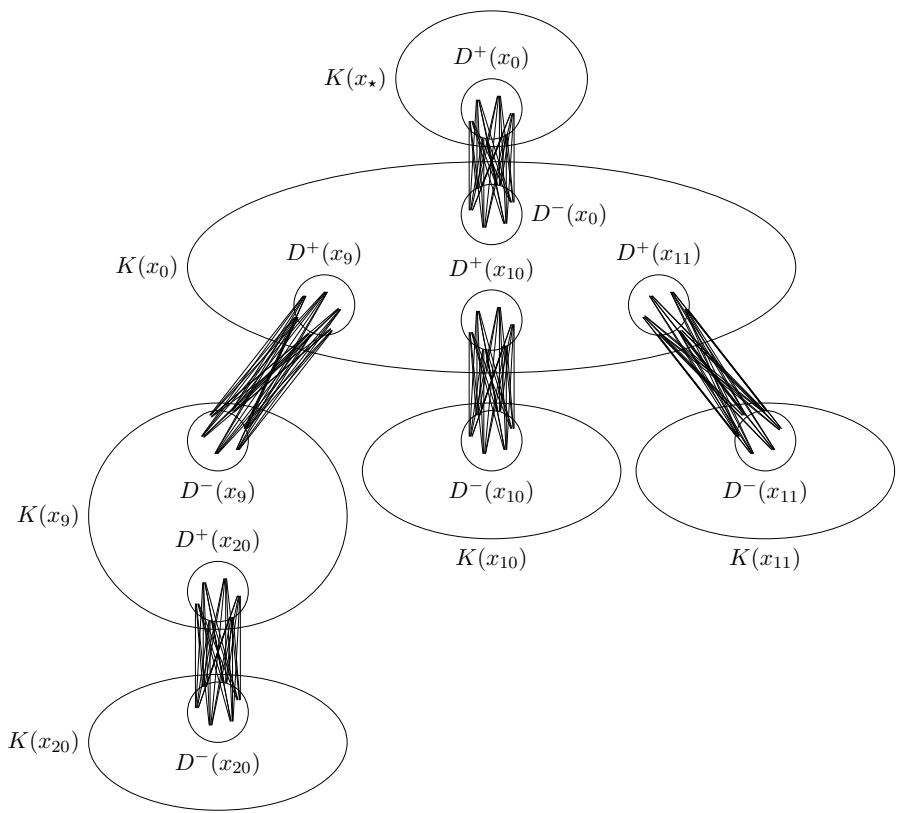

  For each $x \in V(T')\setminus\{x^*\}$, let $D^+(x) = D^{k-1}(x)$ and $D^{-}(x) =
  D^{2k-1}(x)\setminus D^{k-1}(x)$. In order to fix the notation, it helps to
  think in $D^+(x)$ and $D^-(x)$ as the \emph{upper} and \emph{lower half of close
  descendants} of $x$, respectively. We denote by $x^+$ the parent of $x$ in
  $T'$. Suppose that there exists an injective map~$\phi: V(T) \rightarrow
  V(J')$ such that for every $x \in V(T')\setminus\{x^*\}$, we have
  \begin{enumerate}
    \item\label{C1}  $\phi(D^{+}(x)) \subseteq K(x,x^{+}) \cap K(x^+)$;
    \item\label{C2} $\phi(D^{-}(x)) \subseteq K(x,x^{+}) \cap K(x)$.
  \end{enumerate}
  Then we claim that such map is in fact an embedding of $T^k$ into $J'$.
  Figure~\ref{fig1} should help to visualize the concepts developed so far. 

  \begin{claim}\label{claim:phi2}
    If $\phi: V(T)\rightarrow V(J')$ is an injective map such that for all $x
    \in V(T')\setminus\{x^*\}$ the properties~\eqref{C1} and~\eqref{C2} hold,
    then $\phi$ is an embedding of $T^k$ into $J'$.
  \end{claim}
  \begin{claimproof}

    We want to show that if $u$ and $v$ are distinct vertices in $T$ at
    distance at most $k$, then $\phi(u)\phi(v)$ is an edge in $J'$.  Let
    $\tilde{u}$ and $\tilde{v}$ be vertices in $V(T')\setminus\{x^*\}$ with $u
    \in D^{2k-1}(\tilde{u})$ and $v\in D^{2k-1}(\tilde{v})$. If $\tilde{u} =
    \tilde{v}$, then by properties~\eqref{C1} and~\eqref{C2}, we have $\phi(u)$
    and $\phi(v)$ adjacent in $J'$, once all the vertices in
    $\phi(D^{2k-1}(\tilde{u}))$ are adjacent $J'$ either by edges from
    $K(\tilde{u})$, $K(\tilde{u}^+)$ or $K(\tilde{u},\tilde{u}^+)$. If
    $\tilde{u} = \tilde{v}^+$, then we must have $u \in D^-(\tilde{u})$ and $v
    \in D^+(\tilde{v})$ and properties~\eqref{C1} and~\eqref{C2} give us
    $\phi(u),\phi(v) \in K(\tilde{u})$. Analogously if $\tilde{v} =
    \tilde{u}^+$. If $\tilde{u}^+ = \tilde{v}^+$ (with $\tilde{u} \neq
    \tilde{v}$), then we have $u \in D^+(\tilde{u})$ and $v \in D^+(\tilde{v})$
    and property~\eqref{C1} give us $\phi(u),\phi(v) \in K(\tilde{u}^+)$.     

    Therefore we may assume that $\tilde{u}$ and $\tilde{v}$ are at
    distance at least $2$ in $T'$ and do not share a parent.
    But this implies that
    \begin{equation*}
      \min\{\mathrm{dist}_{T}(x,y): x \in D^{2k-1}(\tilde{u}), y \in
      D^{2k-1}(\tilde{v})\} \geq 2k+1,
    \end{equation*}
    contradicting the fact that $u$ and $v$ are at distance at most $k$
    in $T$.
  \end{claimproof}

  We conclude the proof by showing that such a map exists.
  \begin{claim}\label{claim:phi1}
    There is an injective map $\phi: V(T)\rightarrow V(J')$ for which~\eqref{C1}
    and~\eqref{C2} hold for every $x \in V(T')\setminus\{x^*\}$.
  \end{claim}
  \begin{claimproof}
    We just need to show that for every $x \in V(T')$, there is enough room in
    $K(x)$ to guarantee that~\eqref{C1} and~\eqref{C2} hold. In order to do so,
    $K(x)$ should be large enough to accommodate the set
    \begin{equation*}
      D^-(x) \cup \bigcup_{\substack{y \in V(T') \\ y^+=x}} D^+(y).
    \end{equation*}
    Since $T'$ has maximum degree at most $\Delta^{2k}$ and $T$ has maximum
    degree $\Delta$, we have that the set above has at most $\Delta^{4k}$
    vertices.  Finally, once $|K(x)| = \ell \geq r_0 = \Delta^{4k}$, we are
    done. 
  \end{claimproof}

\end{proof}

We end this section discussing a graph property that needs to be inherited by
some subgraphs when running the induction in the proof of
Theorem~\ref{thm:main}.

\begin{definition}
  \label{def:P}
For positive numbers $n$, $a$, $b$, $c$, $\ell$ and $\theta$,
  let $\cP_n(a,b,c, \ell,\theta)$ denote the class of all
 graphs~$G$ with the following properties, where  $p=c/(an)$.
  \begin{enumerate}[label=\rmlabel]
  \item $|V(G)|=an$,\label{def:P-1}
  \item $\Delta(G)\leq b$, \label{def:P-2} 
  \item $G$ has no cycles of length at most $2\ell$,\label{def:P-4}
  \item $G$ is $(p,\theta)$-bijumbled.\label{def:P-5}
  \end{enumerate}
\end{definition}

Only mild conditions on $a$, $b$, $c$, $\ell$ and $\theta$ are necessary
to guarantee the existence of a graph in $\cP_n(a,b,c, \ell,\theta)$ for sufficiently large $n$.
These conditions can be seen in~\ref{goodi}--\ref{goodiii} in Definition~\ref{def:good} below.
In order to keep the induction going in our main proof we also need a condition relating  $k$ and $\Delta$, which represents, respectively, the power of the tree $T$  we want to embed and the maximum degree of $T$ (see \ref{goodiv} in the next definition).

\begin{definition}\label{def:good}
A 7-tuple $(a,b,c,\ell,\theta, \Delta, k)$ is \emph{good} if 
\begin{enumerate}[label=\rmlabel]
\item $a\geq 3$,\label{goodi}
\item $c\geq \theta \ell$,\label{goodii}
\item $b\geq 9c$,\label{goodiii}
\item $\ell \geq 21\Delta^{2k}$.\label{goodiv}
\end{enumerate}
\end{definition}

Next we prove that conditions~\ref{goodi}--\ref{goodiii} in Definition~\ref{def:good} together with $\theta\geq 32\sqrt{c}$ are enough to guarantee that there are graphs in $\cP_n(a,b,c,\ell,\theta)$ as long as $n$ is large enough.
We remark that  next lemma is stated for a good 7-tuple, but condition~\ref{goodiv} of Definition~\ref{def:good} is not necessary and, therefore, also $\Delta$ and $k$ are irrelevant.

\begin{lemma}\label{lem:exist}
If $(a,b,c,\ell,\theta, \Delta, k)$ is a good 7-tuple with $\theta\geq 32\sqrt{c}$, then for sufficiently
large~$n$ the family $\cP_n(a,b,c,\ell,\theta)$ is non-empty.
\end{lemma}

\begin{proof}
Let $(a,b,c, \ell,\theta, \Delta, k)$ be a good 7-tuple with $\theta\geq 32\sqrt{c}$ and let
 $n$ be sufficiently large.
Put $N=an$ and let $G^* = G(3N, p)$ be the binomial random graph with $3N$
vertices and edge probability $p = c/N$.
From Chernoff's inequality (Theorem~\ref{thm:chernoff}) we know that almost
surely
\begin{equation}\label{eq:prob0}
e(G^*)\leq 2p\binom{3N}{2}\leq 9cN.
\end{equation}
From~\cite{HaKoLu95}*{Lemma~8}, we know that a.s. $G^*$ is $(p,
e^{2}\sqrt{6p(3N)})$-bijumbled, i.e.,
the following holds almost surely: for all disjoint sets $X$ and~$Y\subseteq
V(G^*)$ with $|X|\leq|Y|\leq p(3N)|X|$, we have
\begin{equation}\label{eq:prob1}
\big|e_{G^*}(X,Y)-p|X||Y|\big|\leq (e^{2}\sqrt{6})\sqrt{p(3N)|X||Y|}.
\end{equation}

The expected number of cycles of length at most $2\ell$ in $G^*$ is given by
$\mathbb{E}(C_{\leq 2\ell}) = \sum_{i=3}^{2\ell} \mathbb{E}(C_i)$, where $C_i$
is the number of cycles of length $i$.
Then,
\begin{equation*}
  \mathbb{E}(C_{\leq 2\ell}) =  \sum_{i=3}^{2\ell}
  \binom{3an}{i}\frac{(i-1)!}{2}\, p^i \leq 	\sum_{i=3}^{2\ell} (3c)^i \leq
  2\ell (3c)^{2\ell}.
\end{equation*}
Then, from Markov's inequality, we have
\begin{equation}\label{eq:prob2}
\mathbb{P}\big(	C_{\leq 2\ell}\geq  4\ell (3c)^{2\ell} \big) \leq \frac{1}{2}.
\end{equation}
Since~\eqref{eq:prob0} and~\eqref{eq:prob1} hold almost surely and the
probability in~\eqref{eq:prob2} is at most~1/2, for sufficiently large~$n$
there exists a $(p, e^{2}\sqrt{18 c})$-bijumbled graph $G'$ with $3N$ vertices
that contains less than $4\ell (3c)^{2\ell}$ cycles of length at most $2\ell$
and $e(G')\leq 2p\binom{3N}{2}\leq 9cN$.
Then, by removing $4\ell (3c)^{2\ell}$ vertices we obtain a graph $G''$ with no
such cycles such that
\begin{equation*}
  |V(G'')| = 3an - 4\ell (3c)^{2\ell}\geq 2an \qand e(G'')\leq  9cN.
\end{equation*}
To obtain the desired graph $G$ in $\cP_n(a,b,c,\ell,\theta)$, we repeatedly
remove vertices of highest degree in $G''$ until $N$ vertices are left,
obtaining a subgraph $G\subseteq G''$ such that $\Delta(G)\leq 9c\leq b$, as
otherwise we would have deleted more than $e(G'')$ edges.
Note that deleting vertices preserves the bijumbledness.
Therefore, for all disjoint sets $X$ and~$Y\subseteq V(G)$ with $|X|\leq|Y|\leq
p(3N)|X|$ we have
\begin{equation}
\big|e_G(X,Y)-p|X||Y|\big|\leq (e^{2}\sqrt{6})\sqrt{p(3N)|X||Y|}\leq
(32\sqrt{pN})\sqrt{|X||Y|} \leq \theta\sqrt{|X||Y|}.
\end{equation}

We obtained a graph $G$ on $N$ vertices and maximum degree $\Delta(G)\leq b$
such that $G$ contains no cycles of length at most~$2\ell$ and is
$(p,\theta)$-bijumbled, for $p=c/N$.
Therefore, the proof of the lemma is complete.
\end{proof}

\section{Proof of the main result}
\label{sec:main}

We derive Theorem~\ref{thm:main} from Proposition~\ref{prop:induction} below.
Before continuing, given an integer~$\ell\geq1$, let us define what we mean by
a \emph{sheared complete blow-up} $H\{\ell\}$ of a graph~$H$: this is any graph
obtained by replacing each vertex~$v$ in~$V(H)$ by a complete graph~$C(v)$
with~$\ell$ vertices, and by adding all edges \textit{but a perfect matching}
between~$C(u)$ and~$C(v)$, for each $uv\in E(H)$. We also define the
\emph{complete blow-up} $H(\ell)$ of a graph $H$ analogously, but by adding all
the edge between $C(u)$ and $C(v)$, for each $uv\in E(H)$. 

\begin{proposition}
  \label{prop:induction}
For all integers $k \ge 1$, $\Delta\geq 2$, and $s \ge 1$ there exists $r_s$
and a good 7-tuple $(a_s,b_s,c_s,\ell_s,\theta_s, \Delta, k)$ with $\theta_s\geq 32\sqrt{c_s}$
  for which the following holds.  If~$n$ is sufficiently
  large and $G\in\cP_n(a_s,b_s,c_s,\ell_s,\theta_s)$ then, for any tree~$T$
  on~$n$ vertices with $\Delta(T)\leq\Delta$, we have
  \begin{equation*}
    G^{r_s}\{\ell_s\}\rightarrow (T^k)_s.
  \end{equation*}
\end{proposition}

Theorem~\ref{thm:main} follows from Proposition~\ref{prop:induction}
applied to a certain subgraph of a random graph.

\begin{proof}[Proof of Theorem~\ref{thm:main}]
  Fix positive integers $k$, $\Delta$ and $s$ and let $T$ be an
  $n$-vertex tree with maximum degree~$\Delta$.
  Proposition~\ref{prop:induction} applied with parameters~$k$,
$\Delta$ and $s$ gives $r_s$ and a good 7-tuple $(a_s,b_s,c_s, \ell_s,\theta_s,
\Delta, k)$ with $\theta_s\geq 32\sqrt{c_s}$.
  
  Let $n$ be sufficiently large.
By Lemma~\ref{lem:exist}, since $\theta_s\geq 32\sqrt{c_s}$, there exists a
graph $G\in\cP_n(a_s,b_s,c_s,\ell_s,\theta_s)$.
  Let $\chi$ be an arbitrary $s$-colouring of $E(G^{r_s}\{\ell_s\})$.
Then, Proposition~\ref{prop:induction} gives that $G^{r_s}\{\ell_s\}\rightarrow
(T^k)_s$.
Since~$|V(G)|=a_sn$, the maximum degree of $G$ is bounded by the constant
$b_s$, and since~$r_s$ and~$\ell_s$ are constants, we have
  $e(G^{r_s}\{\ell_s\})=O(n)$, which concludes the proof of
  Theorem~\ref{thm:main}.
\end{proof}

The proof of Proposition~\ref{prop:induction} follows by induction in the
number of colours.
Before we give this proof, let us state the results for the base case and the
induction step.

\begin{lemma}[Base Case]
  \label{lem:base}
  For all integers $h\geq 1$, $k\ge1$ and $\Delta \ge 2$ there is an
  integer~$r$ and a good 7-tuple $(a,b,c,\ell, \theta, \Delta, k)$
  with $\theta\geq2^{h-1} 32\sqrt{c}$ such that if $n$ is sufficiently
  large, then the following holds for any
  $G\in\cP_n(a,b,c, \ell,\theta)$.  For any $n$-vertex tree $T$ with
  $\Delta(T)\leq \Delta$, the graph $G^{r}\{\ell\}$ contains a copy of
  $T^k$.
\end{lemma}

\begin{lemma}[Induction Step]
  \label{lem:indstep}
For any positive integers $\Delta\geq 2$, $s\geq 2$, $k$, $r$, $h\geq 1$ and
any good 7-tuple $(a,b,c,\ell,\theta, \Delta, k)$ with $\theta\geq
2^h32\sqrt{c}$, there is a positive integer $r'$ and a good 7-tuple
$(a',b',c',\ell',\theta', \Delta, k)$ with $\theta'\geq 2^{h-1}32 \sqrt{c'}$
such that the following holds.
  If~$n$ is sufficiently large 
then for any graph $G\in\cP_n(a',b',c',\ell',\theta')$ and any
$s$-colouring~$\chi$ of $E(G^{r'}\{\ell'\})$ either
  \begin{enumerate}[label=\rmlabel]
\item there is a monochromatic copy of $T^k$ in $G^{r'}\{\ell'\}$ for any
$n$-vertex tree $T$ with $\Delta(T)\leq \Delta$,
    or\label{def:indstep-1}
  \item there is $H\in\cP_n(a,b,c, \ell,\theta)$ such that
    $H^r\{\ell\} \subseteq G^{r'}\{\ell'\}$ and $H^r\{\ell\}$ is coloured with
    at most $s-1$ colours under~$\chi$.\label{def:indstep-2}
  \end{enumerate}
\end{lemma}

Now we are ready to prove Proposition~\ref{prop:induction}.

\begin{proof}[Proof of Proposition~\ref{prop:induction}]
Fix integers $k \ge 1$, $\Delta\geq 2$ and $s \ge 1$ and define $h_i = s-i$ for
$1\leq i\leq s$.
Let $r_1$ and a good 7-tuple $(a_1,b_1,c_1,\ell_1,\theta_1, \Delta, k)$ with
$\theta_1\geq 2^{h_1} 32\sqrt{c_1}$ be given by Lemma~\ref{lem:base} applied
with $s$, $k$ and $\Delta$.

We will prove the proposition by induction on the number colours with the
additional property that if the colouring has $i$ colours then  $\theta_{i}\geq
2^{h_i} 32\sqrt{c_{i}}$.

Notice that Lemma~\ref{lem:base} implies that for sufficiently large~$n$, if
$G\in\cP_n(a_1,b_1,c_1,\ell_1,\theta_1)$, then $G^{r_1}\{ \ell_1 \} \rightarrow
(T^k)_1 $.
Therefore, since $\theta_1\geq 2^{h_1} 32\sqrt{c_1}$, if $s=1$, we are done.

Assume $s\geq 2$. Suppose the statement holds for $s-1$ colours with the
additional property that $\theta_{s-1}\geq 2^{h_{s-1}} 32\sqrt{c_{s-1}}$,
where
$r_{s-1}$ and a good $7$-tuple
$(a_{s-1},b_{s-1},c_{s-1},\ell_{s-1},\theta_{s-1},\Delta,k)$ are given by
the induction hypothesis.
Therefore, for
any tree~$T$ on~$n$ vertices with $\Delta(T)\leq\Delta$,
we know that for a sufficiently large~$n$
  \begin{equation}
  \label{eq:HarrowsT}
    H^{r_{s-1}}\{\ell_{s-1}\}\rightarrow (T^k)_{s-1} \quad \text{for any} \quad
    H \in \cP_n(a_{s-1},b_{s-1},c_{s-1},\ell_{s-1},\theta_{s-1}).
  \end{equation}

Since $\theta_{s-1}\geq 2^{h_{s-1}} 32\sqrt{c_{s-1}}$, we can apply
Lemma~\ref{lem:indstep} with parameters $\Delta, s, k, r_{s-1}, h_{s-1}$ and
$(a_{s-1},b_{s-1},c_{s-1},\ell_{s-1},\theta_{s-1}, \Delta, k)$, obtaining $r_s$
and $(a_s,b_s,c_s,\ell_s,\theta_s, \Delta, k)$ with $\theta_{s} \geq
2^{h_s}32\sqrt{c_{s}}$.

Let $G\in\cP_n(a_s,b_s,c_s,\ell_s,\theta_s)$ and let $n$ be sufficiently large.
Now let $\chi$ be an arbitrary $s$-colouring of $E(G^{r_s}\{\ell_s\})$.  From
Lemma~\ref{lem:indstep}, we conclude that either~\ref{def:indstep-1} there is a
monochromatic copy of $T^k$ in $G^{r_s}\{\ell_s\}$ for any tree~$T$ on~$n$
vertices with $\Delta(T)\leq\Delta$, in which case the proof is finished,
or~\ref{def:indstep-2} there exists a graph
$H\in\cP_n(a_{s-1},b_{s-1},c_{s-1},\ell_{s-1},\theta_{s-1})$ such that
$H^{r_{s-1}}\{\ell_{s-1}\} \subseteq G^{r_s}\{\ell_s\}$ and
$H^{r_{s-1}}\{\ell_{s-1}\}$ is coloured with at most $s-1$ colours under
$\chi$.  In case~\ref{def:indstep-2}, the induction
hypothesis~\eqref{eq:HarrowsT} implies that we find the desired monochromatic
copy of $T^k$ in $H^{r_{s-1}}\{\ell_{s-1}\} \subseteq G^{r_s}\{\ell_s\}$.
\end{proof}

The proof of Lemma~\ref{lem:base} follows by proving that for a good
7-tuple $(a,b,c,\ell, \theta, \Delta, k)$ with $\theta\geq
2^{h-1} 32\sqrt{c}$, large graphs $G$ in $\cP_n(a,b,c, \ell,\theta)$ are
expanding (using Lemma~\ref{lem:jumb-uni-exp}).
Then, we use Lemma~\ref{lem:FP} to conclude that $G$ contains the desired tree
$T$. After this step we greedily find an embedding of $T^k$ in $G\{\ell\}^k$.

\begin{proof}[Proof of the base case (Lemma~\ref{lem:base})]
Let $h\geq 1$, $k\geq 1$ and $\Delta\geq 2$ be integers.
Let 
\begin{equation*}
  r=k, \quad \ell = 21\Delta^{2k}, \quad \theta = 4^{h} 256 \ell, \quad
  c=\theta \ell, \quad b = 9c
\end{equation*}
and put $D=\Delta + 1$.
Note that $\theta\geq 2^{h-1} 32\sqrt{c}$  and let

\begin{equation*}
  a\geq 2 (D+1) f.
\end{equation*}
Since $\ell\geq 4(\Delta+3)$, we have $c\geq 4(D+2)\theta$.
From the lower bounds on $c$ and $a$ we know that we can use the
conclusion of Lemma~\ref{lem:jumb-uni-exp}  when applied with  
$f=2$, $D=\Delta+1$ and $\theta$.

Note that from our choice of constants, $(a,b,c,\ell,\theta,\Delta,k)$ is a good tuple.
Let $n$ be sufficiently large and let $T$ be a tree on $n$ vertices with
$\Delta(T)\leq \Delta$. Let $G\in\cP_n(a,b,c, \ell,\theta)$. From
Lemma~\ref{lem:jumb-uni-exp} we know that $G$ has an $(n,2,\Delta+1)$-expanding
subgraph and, therefore, from Lemma~\ref{lem:FP} we conclude that $G$ contains
a copy of $T$. Clearly, the graph $G^k$ contains a copy of $T^k$. It remains to
prove that the graph $G^k\{\ell\}$ also contains a copy of $T^k$.

Let $\{v_1,\ldots, v_n\}$ the vertices of $T_n$ and denote by $T_j$ the
subgraph of $T$ induced by $\{v_1,\ldots, v_{j}\}$.
Given a vertex $v\in V(G)$, let $C(v)$ denote the $\ell$-clique in
$G^k\{\ell\}$ that corresponds to $v$.
Suppose that for some $1 \leq j < k$ we have embedded $T_j^k$ in $G^k\{\ell\}$
where, for each $1\leq i \leq j$, the vertex $v_i$ was mapped to some $w_i\in
C(v_i)$.

By the definition of $G^k\{\ell\}$, every neighbour $v$ of $v_{j+1}$ in $G^k$
is adjacent to all but one vertex of $C(v_{j+1})$. Therefore, since
$\Delta(T^k)\leq \Delta^k$ and $|C(v_{j+1})|=\ell \geq \Delta^k+1$, we may thus
find a vertex $w_{j+1}\in C(v_{j+1})$ such that $w_{j+1}$ is adjacent in
$G^k\{\ell\}$ to every $w_i$ with $1\leq i\leq j$ such that $v_iv_{j+1}\in
E(T_{j+1}^k)$. From that we obtain a copy of $T_{j+1}^k$ in $G^k\{\ell\}$ where
$w_i\in C(v_i)$ for $1\leq i\leq j+1$. Therefore, starting with any vertex
$w_1$ in $C(v_1)$, we may obtain a copy of $T^k$ in $G^k\{\ell\}$ inductively,
which proves the lemma.
\end{proof}

The core of the proof of Theorem~\ref{thm:main} is the induction step
(Lemma~\ref{lem:indstep}).
We start by presenting a sketch of its proof.

\begin{proof}[Sketch of the induction step
(Lemma~\ref{lem:indstep})]
We start by fixing suitable constants $r'$, $a'$, $b'$, $c'$, $\ell'$ and
$\theta'$.
Let $n$ be sufficiently large and let~$G\in\cP_n(a',b',c',\ell',\theta')$ be
given.
   Consider an arbitrary
  colouring~$\chi$ of the edges of a sheared complete
  blow-up~$G^{r'}\{\ell'\}$ of~$G^{r'}$ with~$s$ colours.  We shall prove that
  either there is a monochromatic copy of~$T^k$ in~$G^{r'}\{\ell'\}$, or
  there is a graph $H\in\cP_n(a,b,c,\ell,\theta)$ such that a
  sheared complete blow-up $H^r\{\ell\}$ of~$H^r$ is a subgraph
  of~$G^{r'}\{\ell'\}$ and this copy of~$H^r\{\ell\}$ is coloured with at
  most~$s-1$ colours under~$\chi$.

  First, note that, by Ramsey's theorem, if $\ell'$ is large then each
  $\ell'$-clique~$C(v)$ of~$G^{r'}\{\ell'\}$ contains a large monochromatic
  clique.  Let us say that blue is the most common colour of these
  monochromatic cliques. Let these blue cliques be~$C'(v)\subseteq C(v)$.  Then
  we consider a graph $J\subseteq G^{r'}$ induced by the vertices~$v$
  corresponding to the blue cliques~$C'(v)$ and having only the edges~$\{u,v\}$
  such that there is a blue copy of a large complete bipartite graph
  under~$\chi$ in the bipartite graph induced between the blue cliques~$C'(u)$
  and~$C'(v)$ in~$G^{r'}\{ \ell' \}$.  
  
  Then, by Lemma~\ref{lem:alternatives} applied to $J$, either there is a
  set $\emptyset\neq Z\subseteq V(J)$ such that $J[Z]$ is expanding, or
  there are large disjoint sets $V_1,\ldots,V_\ell$ with no edges
  between them in $J$.  In the first case, Lemma~\ref{lem:bluepower}
  guarantees that there is a tree~$T'$ such that, if $T'\subseteq J[Z]$,
  then there is a blue copy of~$T^k$ in $G^{r'}\{\ell'\}$.  To prove that
  $T'\subseteq J[Z]$, we recall that~$J[Z]$ is expanding and use
  Lemma~\ref{lem:FP}.  This finishes the proof of the first case.

  Now let us consider the second case, in which there are large
  disjoint sets $V_1,\ldots,V_\ell$ with no edges between them in~$J$.
  The idea is to obtain a graph $H\in\cP_n(a,b,c,\ell,\theta)$
  such that $H^r\{\ell\} \subseteq G^{r'}\{\ell'\}$ and, moreover,
  $H^r\{\ell\}$ does not have any blue edge.  For that we first obtain
  a path $Q$ in $G$ with vertices $(x_1,\ldots,x_{2a\ell n})$ such
  that $x_i\in V_j$ for all $i$ where $i = j \bmod \ell$.  Then we
  partition~$Q$ into~$2an$ paths $Q_1, \ldots, Q_{2an}$ with~$\ell$
  vertices each, and consider an auxiliary graph~$H'$ on
  $V(H')=\{Q_1,\dots,Q_{2an}\}$ with $Q_iQ_j\in E(H')$ if and only
  $E_G(V(Q_i),V(Q_j))\neq\emptyset$.
  To ensure that $H'$ inherits properties from $G$ we use that there 
  can bet at most one edge between $Q_i$ and $Q_j$ in $G$, because 
  there are no cycles of length less than $2 \ell$ in $G$. 
  
  We obtain a subgraph
  $H''\subseteq H'$ by choosing edges of $H'$ uniformly at random with a
  suitable probability~$p$.  Then, successively removing vertices of
  high degree, we obtain a graph $H\subseteq H''$ with
  $H\in\cP_n(a,b,c,\ell,\theta)$.  It now remains to find a copy of
  $H^r\{ \ell \}$ in $G^{r'}\{\ell'\}$ with no blue edges.  To do so, we first
  observe that the paths~$Q_i\in V(H')$ give rise to $\ell$-cliques
  in~$G^{r'}$ ($r'\geq\ell$).  One can then prove that there is a copy of
$H^r\{\ell\}$ in $G^{r'}$ that avoids the edges of $J$.
By applying the Lov\'{a}sz local lemma we can further deduce that there is a 
copy of $H^r\{\ell\}$ in $H^r\{ \ell \}$ with no blue edges.
  \end{proof}
  
\begin{proof}[Proof of the induction step (Lemma~\ref{lem:indstep})]
We start by fixing positive integers $\Delta\geq 2$, $s\geq 2$, $k$, $r$, $h$
and a good 7-tuple $(a,b,c,\ell,\theta, \Delta, k)$ with
\begin{equation*}
  \theta\geq 2^h32\sqrt{c}.
\end{equation*}
Recall that from the definition of good 7-tuple, we have
\begin{equation*}
  b\geq 9c.
\end{equation*}
Let $d_0$ be obtained from Lemma~\ref{lem:H2} applied with $\ell$ and
$\gamma=1/(2\ell)$ (note that $d_0\leq 10$). Further let 
\begin{align*}
  a'' = \ell ( \Delta^{2k}+2) (2a\cdot d_0+2).
\end{align*}
Notice that 
$a''$ is an upper bound on the value given by Lemma~\ref{lem:alternatives} applied
with $f=2$, $D= \Delta^{2k} +1$, $\ell$ and $\eta = 2 a \cdot d_0$.

Let $r_0$ be given by Lemma~\ref{lem:bluepower} on input $\Delta$
and $k$. 
We may assume $r_0$ is even.
Furthermore, let
\begin{equation*}
  t = \max \{ r_0 , \big(40(\ell b^{r+1}+\ell )\big)^{r_0} \} \qand  \ell' =
  \max\{  4s\ell^2, r_s(t)\},
\end{equation*}
where $r_s(t)=r(t,\dots,t)=r(K_t,\dots,K_t)$ denotes the $s$-colour
Ramsey number for cliques of order~$t$.  Let $a' = \ell' a$ and note
that $a'/s \geq 2a''$
because $\ell \geq
21\Delta^{2k}$.
Define constants $c^*$, $c'$ and $r'$ as follows.
\begin{equation}
    \begin{gathered}
c^*= 2 \ell' c, \quad c' = \frac{\ell'}{2\ell^2} c^* = \frac{\ell'^2}{\ell^2}c,
\quad r'=\ell r.
    \end{gathered}\label{eq:constants}
\end{equation}
Put
\begin{equation*}
  b' = 9c' \qand \theta'  = \frac{c^*}{4 c \ell} \theta = \frac{\ell'}{2\ell}
  \theta
\end{equation*}

\begin{claim}
$(a',b',c',\ell',\theta', \Delta, k)$ is a good 7-tuple and $\theta'\geq
2^{h-1}32\sqrt{c'}$.
\end{claim}

\begin{claimproof}
We have to check all conditions in Definition~\ref{def:good}.
Clearly $a'\geq 3$, $b' \geq 9c'$ and $\ell'\geq \ell \geq 21\Delta^{2k}$.
Below we prove that the other conditions hold
\begin{itemize}
\item $c'\geq \theta'\ell'$:
\begin{equation*}
c' = \frac{\ell'^2}{\ell^2} c  \geq  \frac{\ell'^2}{\ell} \theta = 2\theta'
\ell' > \theta' \ell'.
\end{equation*}
\item $\theta'\geq 2^{h-1} 32 \sqrt{c'}$:
\begin{equation*}
\theta' = \frac{\ell'}{2\ell} \theta \geq \frac{\ell'}{2\ell} 2^h32 \sqrt{c} =
2^{h-1}32 \sqrt{c'}.
\end{equation*}
\end{itemize}
\end{claimproof}

Let $G$ be a graph in $\cP_n(a',b',c',\ell',\theta')$.
Assume 
$$
N_G = a'n \qand p_G=c'/N_G
$$
and let $T$ be an arbitrary tree with $n$ vertices and maximum degree $\Delta$
and consider an arbitrary $s$-colouring  $\chi\colon E(G^{r'}\{\ell'\})\to [s]$
of the edges of $G^{r'}\{\ell'\}$.
We shall prove that either there is a monochromatic copy of~$T^k$
in~$G^{r'}\{\ell'\}$, or there is a graph $H\in\cP_n(a,b,c,\ell,\theta)$ such
that a
sheared complete blow-up $H^r\{\ell\}$ of~$H^r$ is a subgraph
of~$G^{r'}\{\ell'\}$ and this copy of~$H^r\{\ell\}$ is coloured with at
most~$s-1$ colours under~$\chi$.

By Ramsey's theorem (see, for example,~\cite{CoFoSu15}), since $\ell' \ge
r_s(t)$, each $\ell'$-clique~$C(w)$ in~$G^{r'}\{\ell'\}$ (for $w\in V(G)$)
contains a monochromatic clique of size at least $t$. Without lost of
generality, let us assume that most of those monochromatic cliques are blue.
Let $W\subseteq V(G)$ be the set of vertices $w$ such that there is a blue
$t$-clique $C'(w) \subseteq C(w)$. We have 
\begin{equation}\label{eq:Wsize}
  |W|\geq \frac{|V(G)|}{s}=\frac{a'n}{s} \geq 2a''n.
\end{equation} 

Define $J$ as the subgraph of $G^{r'}$ with vertex set $W$ and edge set
\begin{equation*}
  E(J) = \left\{uv\in E(G^{r'}[W]): \text{ there is a blue $K_{r_0,r_0}$
      in $G^{r'}\{\ell'\}[C'(u), C'(v)]$}\right\}.
\end{equation*}
That is, $J$ is the subgraph of $G^{r'}$ induced by $W$ and the edges $uv$ such
that there is a blue copy of~$K_{r_0,r_0}$ under~$\chi$ in the bipartite graph
induced by~$G^{r'}\{\ell'\}$ between the vertex sets of the blue
cliques~$C'(u)$ and~$C'(v)$.   
  
We now apply Lemma~\ref{lem:alternatives} with $f =2$, $D = \Delta^{2k}+1$,
$\ell$, and $\eta = 2a\cdot d_0$ to the graph $J$ (notice that $|V(J)| \geq
2a''n$ is large enough so we can apply Lemma~\ref{lem:alternatives}), splitting
the proof into two cases:
\begin{itemize}
  \item[\ref{prop:blue}] there is $\emptyset\neq Z\subseteq V(J)$ such that
    $J[Z]$ is $(n+1,2,\Delta^{2k} +1)$-expanding,
  \item[\ref{prop:gray}] there exist $V_1,\dots,V_\ell \subseteq V(J)$ such
    that $|V_i| \ge  2ad_0 n$ for $1\leq i\leq  \ell$ and $J[V_i,V_j]$ is empty
    for any $1\leq i<j\leq \ell$. 
\end{itemize}

In case $J[Z]$ is $(n+1,2,\Delta^{2k} +1)$-expanding, we first notice that
Lemma~\ref{lem:bluepower} applied to the graph $J[Z]$ implies the existence of
a tree
  $T^\prime= T^\prime(T,\Delta,k)$ of maximum degree $\Delta^{2k}$
with at most $n+1$ vertices such that if $J[Z]$ contains $T^\prime$, then
$T^k\subseteq J^\prime$ for any
  $(r_0,r_0)$-blow-up $J^\prime$ of $J$.
But since $J[Z]$ is $(n+1,2,\Delta^{2k} +1)$-expanding, Lemma~\ref{lem:FP}
implies that $J[Z]$ contains a copy of $T^\prime$.
Therefore, the graph $G^{r'}\{\ell'\}$ contains a blue copy of $T^k$, as we can
consider $J'$ as the subgraph of $G^{r'}\{\ell'\}$ containing only edges inside
the blue cliques $C'(u)$ (which have size $t \ge r_0$) and the edges of the
complete blue bipartite graphs $K_{r_0,r_0}$ between the blue cliques $C'(u)$.
   This finishes the proof of the first case.

	We may now assume that there are subsets $V_1,\dots,V_\ell \subseteq V(J)$ with
	$|V_i| \geq  2ad_0 n$ for $1\leq i\leq  \ell$ and $J[V_i,V_j]$ is empty for
    any $1\leq i<j\leq \ell$.
    We want to obtain a graph $H\in\cP_n(a,b,c,\ell,\theta)$
  	such that $H^r\{\ell\} \subseteq G^{r'}\{\ell'\}$ and contains no blue edges.
  	
   	Let $J'=J[V_1\cup\dots\cup V_\ell]$, $G'=G[V_1\cup\dots\cup V_\ell]$ and note that $|V(G')| = |V(J')| \geq d_0\cdot 2a\ell  n$,
	where we recall that $d_0$ is the constant obtained by applying Lemma~\ref{lem:H2}
	with $\ell$ and $\gamma=1/(2\ell)$.
	We want to use the assertion of Lemma~\ref{lem:H2} to obtain a transversal path of length $2a \ell n$ in $G'$ and so we have to check the conditions adjusted to this parameter.

First note, that we have $|V_i| \geq  2ad_0n \ge \gamma d_0 \cdot 2 a \ell n$ for  $1 \leq i \leq \ell$.
Moreover, since $G'$ is an induced subgraph of $G$ and
$G\in\cP_n(a',b',c',\ell,\theta')$, we know by~\eqref{def:P-sets} that for all
$X,Y\subseteq V(G')$ with $|X|, |Y| > \theta' a'n/c'$ we have
$e_{G'}(X,Y)>0$.  Observe that $\theta' a'n/c' < an = \gamma \cdot 2 a \ell n$ once $a' = \ell' a$ and
$c' > \theta' \ell'$.
Therefore, we may use Lemma~\ref{lem:H2} to conclude that $G'$ contains a path
$P_{2a\ell n}$ on $(x_1,\dots,x_{2a\ell n})$ with $x_i \in V_j$ for all~$i$,
where $j\equiv i\pmod{\ell}$.

We split the obtained path $P_{2a\ell n}$ of $G'$ into consecutive paths $Q_1,
\ldots, Q_{2an}$ each on $\ell$ vertices.
More precisely, we let $Q_i=(x_{(i-1)\ell +1},\ldots,x_{i\ell })$ for
$i=1,\ldots,2an$.
The following auxiliary graph is the base of our desired graph
$H\in\cP_n(a,b,c,\ell,\theta)$.
\begin{align*}
&\text{$H'$ is the graph on $V(H')=\{Q_1, \ldots, Q_{2an}\}$ such that
$Q_iQ_j\in E(H')$ if and only if}\\
&\text{there is an edge in $G$ between the vertex sets of $Q_i$ and $Q_j$}.
\end{align*}

\begin{claim}\label{claim:Hp}
$H'\in \cP_n(2a,\ell b',c^*,\ell,\ell\theta')$.
\end{claim}
\begin{claimproof}
We verify the conditions of Definition~\ref{def:P}.
Since $H'$ has $2an$ vertices, condition~\ref{def:P-1} clearly holds.
Since $\Delta(G)\leq b'$ and for any $Q_i\in V(H')$ we have $|Q_i|=\ell$ (as a
subset of $V(G)$), there are at most $\ell b'$ edges in $G$ with an endpoint in
$Q_i$.
Then, $\Delta(H')\leq \ell b'$.

For condition~\ref{def:P-4}, recall that any vertex of $H'$ corresponds to a
path on~$\ell$ vertices in $G$.
Thus, a cycle of length at most $2\ell$ in $H'$ implies the existence of a
cycle of length at most $2\ell^2$ in $G$.
Since $2\ell'\geq 2\ell^2$ and $G$ has no cycles of length at most $2\ell'$, we
conclude that $H'$ contains no cycle of length at most $2\ell$, which verifies
condition~\ref{def:P-4}.

Let us verify condition~\ref{def:P-5}.  Let $N_{H'}=2an$ and 
\begin{align}\label{eq:pcomc}
p_{H'}=\frac{c^*}{N_{H'}} = \frac{c^*}{2an}.
\end{align}
Consider sets $X=\{Q_1,\ldots,Q_{x}\}$ and let $Y=\{Q_{x+1},\ldots, Q_{x+y}\}$,
where $|X|\leq |Y|\leq p_{H'}N_{H'}|X|$.  Let
${X_G}=\bigcup_{j=1}^{x}Q_j\subseteq V(G)$ and
${Y_G}=\bigcup_{j=x+1}^{x+y}Q_j\subseteq V(G)$.  Note that $|X_G|=\ell|X|$ and
$|Y_G|=\ell |Y|$. As there are no cycles of length smaller than $2 \ell$
in $G$, we only have at most one edge between the vertex sets of $Q_i$ and
$Q_j$. Therefore we have
\begin{equation}\label{eq:edgeH}
e_{H'}({X},{Y})=e_G({X_G},{Y_G}).
\end{equation}
We shall prove that $\left|e_{H'}(X,Y) - p_{H'}|X||Y|\right|\leq \ell
\theta'\sqrt{|X||Y|}$. From the choice of $c'$, we have
\begin{equation}\label{eq:p_H}
p_{H'}|X||Y| =  \frac{c^*}{2an}|X||Y| = \frac{c'}{a'n}\ell|X|\ell|Y| =
\frac{c'}{a'n}|X_G||Y_G| = p_G|X_G||Y_G|.
\end{equation}
Note that since $\ell'\geq 2\ell^2$, we have $c^*\leq c'$.  Then, from $|X|\leq
|Y|\leq p_{H'}N_{H'}|X|$ and $c^* \le c'$ we obtain
\begin{equation*}
  |X_G|\leq |Y_G| \leq p_{G}N_G|X_G|.
\end{equation*}

Combining~\eqref{eq:p_H} with~\eqref{eq:edgeH} and the fact that $G$ is
$(p_{G},\theta')$-bijumbled, we get that
\begin{equation}\label{eq:edgeH'}
|e_{H'}(X,Y) - p_{H'}|X||Y|| = |e_{G}(X_G,Y_G) - p_{G}|X_G||Y_G||\leq
\theta'\sqrt{|X_G||Y_G|} = \ell\theta'\sqrt{|X||Y|}.
\end{equation}
Therefore, $H'$  is $(p_{H'},\ell\theta')$-bijumbled, which verifies
condition~\ref{def:P-5}.
\end{claimproof}

The parameters for $\cP_n(2a,\ell b',c^*,\ell,\ell\theta')$ are tightly fitted
such that we can find the following subgraph of $H'$.

\begin{claim}\label{cl:new}
There exists $H \subseteq H'$ such that $H\in \cP_n(a,b,c,\ell,\theta)$.
\end{claim}

\begin{claimproof}
We first obtain $H''\subseteq H'$ by picking each edge of $H'$ with probability

$$
p=\frac{2c}{c^*}=\frac{1}{\ell'}
$$
independently at random.
Note that $p\leq 1/2$.

From Proposition~\ref{prop:jumbled}, we get 
\begin{align*}
e(H') \le p_{H'} \binom{2an}{2} + \ell \theta' 2 an \le (c^*+2 \ell \theta') an
\le (c^*+ 2\ell \frac{c'}{\ell'}) an \le 2 c^* an
\end{align*}
From Chernoff's inequality, we then know that almost surely we have
\begin{align}
\label{Eq: eH''}
e(H'') \leq  2 p\cdot e(H') \leq  2 \cdot \left(\frac{2c}{c^*}\right) \cdot 2
c^* an \leq 8 acn \leq abn.
\end{align}

Let $N_{H''} = 2an$ and
\begin{equation*}
  p_{H''} = p\cdot p_{H'} = \frac{c}{an}.
\end{equation*}
We will prove that $H''$ is $(p_{H''},\theta)$-bijumbled almost surely. Note
that, by using Chernoff's inequality~(Theorem~\ref{thm:chernoff}) for some
$\eps\leq \min(\ell\theta'/(4\ell'c^2), 1/2)$, we may ensure that for any $X$
and $Y$ disjoint subsets of vertices of $H''$ with $|X|\leq |Y|\leq p_{H''}
N_{H''} |X|$, we have 
\begin{equation}
\begin{aligned}
\label{eq:H''XY}
e_{H''}(X,Y) 	&=(1\pm \eps)p\cdot e_{H'}(X,Y)\\
			&=p\cdot e_{H'}(X,Y) \pm \eps p\cdot e_{H'}(X,Y).
\end{aligned}
\end{equation}

Observe that $p_{H''} N_{H''} |X| = 2c|X| \leq c^*|X| = p_{H'}N_{H'}|X|$.
Therefore, $H''$ is almost surely $(p_{H''},\theta)$-bijumbled, as
by~\eqref{eq:edgeH'} and~\eqref{eq:H''XY} we get
\begin{align*}
  \left|e_{H''}(X,Y) - p_{H''} |X| |Y|\right| 
  & \leq \left|e_{H''}(X,Y) - p\cdot e_{H'}(X,Y)\right| + \left| p\cdot
  e_{H'}(X,Y) - p_{H''}|X||Y| \right| \\
  & \overset{\text{\eqref{eq:H''XY}}}{\leq} \varepsilon p\cdot e_{H'}(X,Y) +
  p\cdot \left| e_{H'}(X,Y) - p_{H'}|X||Y| \right| \\
  & \overset{\text{\eqref{eq:edgeH'}}}{\leq} \varepsilon p\cdot \left(p_{H'}|X||Y|
  + \ell\theta'\sqrt{|X||Y|} \right) + p\cdot \ell \theta'\sqrt{|X||Y|} \\
  & = \varepsilon p_{H''}|X||Y| + (1+\varepsilon)p\cdot \ell\theta'\sqrt{|X||Y|} \\
  & = \frac{\varepsilon c}{an} |X||Y| + (1+\varepsilon)\frac{\ell
  \theta'}{\ell'}\sqrt{|X||Y|} \\
  & \leq 2\frac{\ell \theta'}{\ell'}\sqrt{|X||Y|} \\
  & = \theta\sqrt{|X||Y|}.
\end{align*}
The last inequality follows from the choice of $\varepsilon$ and from the fact that $|X|,|Y|\leq 2an$.

Now we construct the desired graph $H$ from $H''$ by sequentially removing the
$an$ vertices of highest degree.  Notice that $H$ has maximum degree at most
$b$, otherwise this would imply that $H''$  has more than $abn$ edges,
contradicting~\eqref{Eq: eH''}.  Since $H$ is a subgraph of $H'$, and $H'$ does
not contain cycles of length at most $2\ell$, the same holds for $H$.  Finally,
since deleting vertices preserves the bijumbledness property, we conclude that
$H\in \cP_n(a,b,c,\ell,\theta)$.
\end{claimproof}

Recall that $J$ is the subgraph of $G^{r'}$ induced by $W$, with $|W|\geq
a'n/s$ and edges $uv$ such that there is a blue copy of~$K_{r_0,r_0}$
under~$\chi$ in the bipartite graph induced by the vertex sets of blue
cliques~$C'(u)$ and~$C'(v)$ in~$G^{r'}\{\ell'\}$.  Furthermore, recall that
there are subsets $V_1,\dots,V_{\ell} \subseteq V(J)$ with $|V_i| \geq 2ad_0 n$
for $1\leq i\leq  \ell$ and $J[V_i,V_j]$ is empty for any $1\leq i<j\leq \ell$,
and we defined $J'=J[V_1\cup\dots\cup V_{\ell}]$ and $G'=G[V_1\cup\dots\cup
V_{\ell}]$.  Lastly, recall that $Q_i=(x_{(i-1)\ell+1},\ldots,x_{i\ell})$ for
$i=1,\ldots,2an$, where the vertices $x_i$ belong to $G'$.  Assume, without
loss of generality, $V(H) = \{Q_1, \ldots, Q_{an}\}$.  In what follows, when
considering the graph $H^{r}(\ell)$, the $\ell$-clique corresponding to $Q_i$
is composed of the vertices $x_{(i-1)\ell+1},\ldots,x_{i\ell}$, and hence one
can view $V\big(H^{r}(\ell)\big)$ as a subset of $V(G')$.

\begin{claim}\label{claim:Hblow}
$H^{r}(\ell)\subseteq G^{r'}$.
Moreover, $G^{r'}$ contains a copy of $H^{r}\{\ell\}$ that avoids the edges of
$J$.
\end{claim}

\begin{claimproof}

We will prove that $H^{r}(\ell)\subseteq G^{r'}$ where $Q_1, \ldots,
Q_{an}\subseteq{V(J)}$ are the $\ell$-cliques of $H^{r}(\ell)$.
Suppose that $Q_i$ and $Q_j$ are at distance at most $r$ in the graph $H$.
Without loss of generality, let $Q_i=Q_1$ and $Q_j=Q_m$ for some $m\leq r$.
Moreover, let $(Q_1,Q_2,\ldots, Q_m)$ be a path in $H$.
Note that there exist vertices $u_1,\ldots, u_{m-1}$ and $u_2',\ldots, u_m'$ in
$V(G')$ such that $u_1\in Q_1$, $u_m'\in Q_m$, $u_j, u_j'\in Q_j$ for all
$j=2,\ldots, {m-1}$ and $\{u_i, u_{i+1}'\}$ is an edge of $G'$ for $i=1,\ldots,
m-1$.

Let $u_1'\in Q_1$ and $u_m\in Q_m$ be arbitrary vertices.
Since for any $j$, the set $Q_j$ is spanned by a path on $\ell$ vertices in
$G'$, it follows that $u_j$ and $u_j'$ are at distance at most $\ell-1$ in $G'$
for all $1\leq j\leq m$.
Therefore, $u_1'$ and $u_m$ are at distance at most $(\ell-1)m+(m-1)< \ell r
\leq r'$ in $G'$ and hence $u_1'u_m$ is an edge in $G[V_1 \cup \ldots \cup
V_\ell]^{r'} \subseteq G^{r'}$.
Since the vertices $u_1'$ and $u_m$ were arbitrary, we have shown that if $Q_i$
and $Q_j$ are adjacent in $H^{r}$ (i.e., $Q_i$ and $Q_j$ are at distance at
most $r$ in $H$) then $(Q_i,Q_j)$ gives a complete bipartite graph $C(Q_i,
Q_j)$ in $G^{r'}$.
Moreover, taking $i=j$ we see that each~$Q_i$ in~$G^{r'}$ must be complete. 
This implies that $H^{r}(\ell)$ is a subgraph of $G^{r'}$.
 
For the second part of the claim we consider which of the edges of this copy of
$H^{r}(\ell)$ can also be edges of $J$.  Recall from the definition of $J'$
that we found subsets $V_{1}, \ldots, V_{\ell}\subseteq J$ such that no edge of
$J$ lies between different parts.  Moreover each set $Q_i\subseteq J$ takes
precisely one vertex from each set $V_{1}, \ldots, V_{\ell}$.  It follows that
each $Q_i$ is independent in~$J$.  Now let us say we have $x\in Q_i$ and $y\in
Q_j$ ($i\neq j$) that are adjacent in $J$. We can not have  $x$ and $y$
different parts of the partition $\{V_1,\ldots, V_{\ell}\}$. Thus $x$ and $y$
lie in the same part. Therefore edges from $J$ between $Q_i$ and $Q_j$ must
form a matching. Then we can find a copy of $H^{r}\{\ell\}$ that avoids $J$ by
removing a matching between the $l$-cliques from $H^r(\ell)$.

\end{claimproof}

To complete the proof of Lemma~\ref{lem:indstep}, we will embed a copy of the
graph $H^{r}\{\ell\}\subseteq G^{r'}$ found in Claim~\ref{claim:Hblow} in
$G^{r'}\{\ell'\}$ in such a way that $H^{r}\{\ell\}$ uses at most $s-1$
colours.

\begin{claim}\label{claim:LLL}
$G^{r'}\{\ell'\}$ contains a copy of $H^{r}\{\ell\}$ with no blue edges.
\end{claim}

\begin{claimproof}
Recall that each vertex $u$ in $J$ corresponds to a clique $C'(u)\subseteq
G^{r'}\{\ell'\}$ of size $t$ and that this clique is monochromatic in blue in
the original colouring $\chi$ of $E(G^{r'}\{\ell'\})$.
Recall also that if an edge $\{u,v\}$ of $G^{r'}[W]$ is not in $J$, then there
is no blue copy of $K_{r_0,r_0}$ in the bipartite graph between $C'(u)$ and
$C'(v)$ in $G^{r'}\{\ell'\}$. By the K\H{o}v\'ari--S\'{o}s--Tur\'{a}n theorem
(Theorem~\ref{thm:kst}), there are at most $4t^{2-1/r_0}$ blue edges between
$C'(u)$ and $C'(v)$. Recall further that $C'(u)$ and $C'(v)$ are, respectively,
subcliques of the $\ell'$-cliques $C(u)$ and $C(v)$ in $G^{r'}\{\ell'\}$.
Since $\{u,v\}$ is an edge of $G^{r'}$, there is a complete bipartite graph
with a matching removed between $C(u)$ and $C(v)$ in $G^{r'}\{\ell'\}$ and so
there is a complete bipartite graph with at most a matching removed for $C'(u)$
and $C'(v)$.
It follows that there are at least
\begin{equation*}
t^2-t-4t^{2-1/r_0}
\end{equation*}
non-blue edges between $C'(u)$ and $C'(v)$. 

Using the copy of $H^{r}\{\ell\}\subseteq G^{r'}$ avoiding edges of $J$
obtained in Claim~\ref{claim:Hblow} as a `template', we will embed a copy of
$H^{r}\{\ell\}$ in $G^{r'}\{\ell'\}$ with no blue edges.
For each vertex $u\in V(H^{r}\{\ell\})\subseteq V(J)$ we will pick precisely
one vertex from $C'(u)\subseteq G^{r'}\{\ell'\}$ in our embedding.
The argument proceeds by the Lov\'asz Local Lemma.

For each $u\in V(H^r\{\ell\})\subseteq V(J)$ let us choose~$x_u\in C'(u)$
uniformly and independently at random.
Let~$e=\{u,v\}$ be an edge of our copy of $H^{r}\{\ell\}$ in $G^{r'}$ that is
not in $J$.
As pointed out above, we know that there are at least $t^2-t-4t^{2-1/r_0}$
non-blue edges between $C'(u)$ and $C'(v)$.
Letting~$A_e$ be the event that~$\{x_u,x_v\}$ is a blue edge or a non-edge
in~$G^{r'}\{\ell'\}$, we have that
\begin{equation*}
\PP[A_e]\leq\frac{t+4t^{2-1/r_0}}{t^2}\leq5t^{-1/r_0}.
\end{equation*}

  The events~$A_e$ are not independent, but we can define a dependency
  graph~$D$ for the collection of events~$A_e$ by adding
  an edge between~$A_e$ and~$A_f$ if and only
  if~$e\cap f\neq\emptyset$.
  Then, $\Delta=\Delta(D)\leq2\Delta(H^{r}\{\ell\})\leq2(b^{r+1}\ell+\ell)$.
   From our choice of $t$ we get that
  \begin{equation*}
    \label{eq:4epd}
    4\Delta\PP[A_e]\leq40(b^{r+1}\ell+\ell^2)t^{-1/r_0}\leq1
  \end{equation*}
  for all~$e$. Then the Local Lemma tells us
  that~$\PP\big[\bigcap_{e}\bar A_e\big]>0$, and hence a
  simultaneous choice of the~$x_u$'s ($u\in V(H^{r}\{\ell\})$) is possible,
  as required.  This concludes the proof of Claim~\ref{claim:LLL}.
\end{claimproof}
The proof of Lemma~\ref{lem:indstep} is now complete.
\end{proof}

\section{Concluding remarks}
\label{sec:concluding}

To construct our graphs we need that $\cP_n(a,b,c,e,\theta)$ is
non-empty given a good $7$-tuple $(a,b,c,\ell,\theta,\Delta,k)$ with
$\theta \ge 32 \sqrt{c}$.  We prove this in Lemma~\ref{lem:exist}
using the binomial random graph.  Alternatively, it is possible to
replace this by using explicit constructions of high girth expanders.
For example, the Ramanujan graphs constructed by Lubotzky, Phillips,
and Sarnak~\cite{LPS_ramanujan} can be used to prove
Lemma~\ref{lem:exist}.

We now discuss further connections between powers of trees and graph
parameters related to treewidth.  As pointed out in the introduction,
every graph with maximum degree and bounded treewidth is contained in
some bounded power of a bounded degree tree and vice versa.  This
implies that Corollary~\ref{cor:main} is equivalent to
Theorem~\ref{thm:main}.  For bounded degree graphs, bounded treewidth
is equivalent to bounded cliquewidth and also to bounded
rankwidth~\cite{KLM_cliquewidth}.  Therefore, Corollary~\ref{cor:main}
also holds with treewidth replaced by any of these parameters.
Finally, an obvious direction for further research is to investigate
the size-Ramsey number of powers~$T^k$ of trees~$T$ when~$k$
and~$\Delta(T)$ are no longer bounded.

\def\MR#1{}

\end{document}